\documentclass[11pt]{article}
\usepackage[utf8]{inputenc}
\usepackage{amsmath, latexsym, amsfonts, amssymb, amsthm}
\usepackage{graphicx, color,hyperref,dsfont,epsfig,caption,wrapfig,subfig}
\usepackage{pstricks}
\usepackage{datetime}
\usepackage{movie15}
\hypersetup{
    linktoc=page,
    linkcolor=red,          
    citecolor=blue,        
    filecolor=blue,      
    urlcolor=cyan,
    colorlinks=true           
}

\newcommand{\cA}{{\ensuremath{\mathcal A}} }
\newcommand{\cB}{{\ensuremath{\mathcal B}} }
\newcommand{\cO}{{\ensuremath{\mathcal O}} }
\newcommand{\cF}{{\ensuremath{\mathcal F}} }

\newcommand{\bE}{{\ensuremath{\mathbf E}} }

\newcommand{\bbC}{{\ensuremath{\mathbb C}} }

\newcommand{\bbE}{{\ensuremath{\mathbb E}} }

\newcommand{\bbL}{{\ensuremath{\mathbb L}} }

\newcommand{\bbP}{{\ensuremath{\mathbb P}} }

\newcommand{\bbR}{{\ensuremath{\mathbb R}} }

\newcommand{\gb}{\beta}
\newcommand{\gep}{\varepsilon}       
\newcommand{\gG}{\Gamma}
\newcommand{\gL}{\Lambda}
\numberwithin{equation}{section}

\newtheorem{theorem}{Theorem}[section]
\newtheorem{lemma}[theorem]{Lemma}
\newtheorem{proposition}[theorem]{Proposition}
\newtheorem{corollary}[theorem]{Corollary}

\newtheorem{remark}[theorem]{Remark}

\theoremstyle{definition}

\renewcommand{\tilde}{\widetilde}          
\DeclareMathSymbol{\leqslant}{\mathalpha}{AMSa}{"36} 
\DeclareMathSymbol{\geqslant}{\mathalpha}{AMSa}{"3E} 
\DeclareMathSymbol{\eset}{\mathalpha}{AMSb}{"3F}     
\renewcommand{\leq}{\;\leqslant\;}                   
\renewcommand{\geq}{\;\geqslant\;}                   
\newcommand{\dd}{\text{\rm d}}             

\newcommand{\C}{\mathbb{C}}
\newcommand{\R}{\mathbb{R}}

\newcommand{\N}{\mathbb{N}}

\newcommand{\E}{\mathds{E}}
\renewcommand{\P}{\mathds{P}}
\newcommand{\ind}{\mathds{1}}

\usepackage{xparse}

\DeclareDocumentCommand \Pmp { m m o} {
\IfNoValueTF{#3}
{P_{#1}^{#2}}
{P_{#1}^{#2}\left(#3\right)}
}

\DeclareDocumentCommand \Emp { m m o} {
\IfNoValueTF{#3}
{E_{#1}^{#2}}
{E_{#1}^{#2}\left[#3\right]}
}

\DeclareDocumentCommand \Pbr { m m m m o } {
\IfNoValueTF{#5}
{P_{#1}^{#2\stackrel{#4}{\rightarrow} #3}}
{P_{#1}^{#2\stackrel{#4}{\rightarrow} #3}\left(#5\right)}
}

\DeclareDocumentCommand \Ebr { m m m m o } {
\IfNoValueTF{#5}
{E_{#1}^{#2\stackrel{#4}{\rightarrow} #3}}
{E_{#1}^{#2\stackrel{#4}{\rightarrow} #3}\left[#5\right]}
}

\newcommand{\cov}{\mathrm{Cov}}
\newcommand{\var}{\mathrm{Var}}

\def\bi{\begin{itemize}}
\def\ei{\end{itemize}}

\def\bnum{\begin{enumerate}}
\def\enum{\end{enumerate}}
\def\<#1{\langle #1 \rangle}

\def\cF{\mathcal{F}}

\title{The semiclassical limit of Liouville conformal field theory}

\author{ Hubert Lacoin \footnote{IMPA, Rio de Janeiro, Brasil.}, R\'emi Rhodes \footnote{Universit{\'e} Aix-Marseille, I2M, Marseille, France. Partially supported by grant ANR-15-CE40-0013  Liouville.} \footnotetext[2]{Partially supported by grant ANR-15-CE40-0013  Liouville.},
 Vincent Vargas \footnote{ENS Ulm, DMA, 45 rue d'Ulm,  75005 Paris, France. Partially supported by grant ANR-15-CE40-0013  Liouville.} }

\begin{document}

\maketitle

\begin{abstract}
A rigorous probabilistic construction of Liouville conformal field theory (LCFT) on the Riemann sphere was recently given by David-Kupiainen and the last two authors. In this paper, we focus on the connection between LCFT and the classical Liouville field theory via the semiclassical approach. LCFT depends on a parameter $\gamma \in (0,2)$ and the limit $\gamma \to 0$ corresponds to the semiclassical limit of the theory. Within this asymptotic and under a negative curvature condition (on the limiting metric of the theory), we determine the limit of the correlation functions and of the associated Liouville field. We also establish a large deviation result for the Liouville field: as expected, the large deviation functional is the classical Liouville action.  As a corollary, we give a new (probabilistic) proof of the Takhtajan-Zograf theorem which relates the classical Liouville action (taken at its minimum) to Poincar\'e's accessory parameters. Finally, we gather conjectures in the positive curvature case (including the study of the so-called quantum spheres introduced by Duplantier-Miller-Sheffield).   
\end{abstract}

\begin{center}
\end{center}
\footnotesize



\noindent{\bf Key words or phrases:}  Liouville Quantum Theory, Gaussian multiplicative chaos, Polyakov formula, uniformization, accessory parameters, semiclassical analysis.


\normalsize

\vspace{0.1cm}
\tableofcontents


\section{Introduction}
%
The purpose of this paper is to relate classical Liouville theory to Liouville conformal field theory (also called quantum Liouville theory)
through a semiclassical analysis. 
Before exposing the framework and the results of the paper,  we provide a short historical introduction to classical Liouville theory.

\medskip

We start by recalling a classical result of  uniformization by Picard about   the existence of hyperbolic metrics on the Riemann sphere, seen as the extended complex plane $\hat\C=\C\cup\{\infty\}$, with prescribed conical singularities.   In this context,  given $n\ge 3$ we let  
 $z_1,\dots,z_n\in \C $ distinct and $\chi_1,\dots,\chi_n\in \bbR$ denote respectively the prescribed locations and orders of our singularities. 
 We assume that our coefficients satisfy the two following conditions:
\begin{equation}\label{Negcurv}
\forall k, \quad   \chi_k <2 \quad \text{ and } \quad     \sum_{k=1}^n\chi_k>4.
\end{equation} 
Let $\Delta_z$ denote the standard Laplacian on $\C$, $\nabla_z$ the standard gradient and for $z=x+iy$ we adopt the complex notation $\partial_z=\frac{1}{2}(\partial_x-i \partial_y)$. We use the notation 
 $\frac{d}{d z}$ for the complex derivative of a holomorphic function.
 The classical result of Picard \cite{picard1,picard2} (see also \cite{troyanov}) asserts that for any $\Lambda>0$ the Liouville equation 
\begin{equation}\label{Liouv}
\Delta_z \phi=2 \pi \Lambda e^\phi
\end{equation}
possesses a unique smooth solution $\phi$ on $\C\setminus \{z_1,\dots,z_n\}$ with the following asymptotics near the singular points and  at infinity
\begin{equation}\label{lcfasympt}
\begin{cases}
 \phi(z)=\chi_k\ln\frac{1}{|z-z_k|}+\mathcal{O}(1)\quad &\text{as }\quad  z\to z_k,\\
  \phi(z)=-4\ln |z|+\mathcal{O}(1)\quad &\text{as }\quad  z\to \infty.
\end{cases}
  \end{equation}
We denote this solution by $\phi_\ast$.  The first condition in \eqref{Negcurv} simply ensures that $e^{\phi_{*}}$ is integrable in a neighborhood of each $z_k$ so that $e^{\phi_\ast(z)}|dz|^2$ defines indeed a compact metric on $\hat \bbC$. The second condition 
\begin{equation}\label{realNegcurv}
\sum_{k=1}^n\chi_k>4 
\end{equation}
is a negative curvature type condition. 
In the language of Riemannian geometry, the metric $e^{\phi_\ast(z)}|dz|^2$ has constant negative  Ricci scalar  curvature (as this is the only notion of curvature used in this paper, we simply refer to to it as curvature in the remainder of the text) equal to $-2\pi\Lambda$ on $\hat\C\setminus \{z_1,\dots,z_n\}$ (recall that the curvature of the metric at $z$ is given by $e^{-\phi_\ast(z)}\Delta_z \phi_\ast(z)$) and a conical singularity of order $\chi_k$ at $z_k$ for each $k=1,\dots, n$. Using standard integration by parts (on the Riemann sphere equipped with the round metric), it is possible to show that the solution $\phi_\ast$ to the system \eqref{Liouv}-\eqref{lcfasympt} must satisfy 
\begin{equation}\label{Gauss}
\sum_{k=1}^n\chi_k-4 =  \Lambda   \left (  \int_\C  e^{\phi_\ast(z)} \dd^2z \right ). 
\end{equation} 
where $ \dd^2z$ denotes the standard Lebesgue measure. Therefore $\sum_{k=1}^n\chi_k-4 $ and $\Lambda$ have same sign hence justifying the terminology \emph{negative curvature type condition}  for the condition \eqref{realNegcurv}. In a celebrated work \cite{rondalors}, Poincar\'e showed how to relate this metric to the problem of the uniformization of $\hat\C \setminus \lbrace z_1, \cdots, z_n \rbrace$\footnote{In fact, Poincar\'e considered the case $\chi_k=2$ for all $k$.}. More specifically, he introduced the $(2,0)$-component of the ``classical" stress-energy tensor     
\begin{equation}
T_{\phi_\ast}(z)=\partial^2_{zz}\phi_\ast(z)-\tfrac{1}{2}(\partial_z \phi_\ast (z))^2.
\end{equation}
Direct computations show that \eqref{Liouv} implies that $T_{\phi_\ast}$ is a meromorphic function on $\hat\C$ and then \eqref{lcfasympt} implies that it displays  second order poles at $z_1,\dots,z_n$. More precisely we must have
\begin{equation}\label{merostress}
T_{\phi_\ast}(z)=\sum_{k=1}^n\Big(\frac{\chi_k/2-\chi_k^2/8}{(z-z_k)^2}+\frac{c_k}{z-z_k}\Big)
\end{equation}
with  the asymptotics $T_{\phi_\ast}(z)= \mathcal{O}(z^{-4})$ as  $z\to\infty$  where the  real numbers $c_k$ are the so-called {\it accessory parameters}. Then, considering the second order  Fuchsian equation for holomorphic  functions on the universal cover (here the unit disk)
\begin{equation}\label{Fuchs}
\frac{d^2 u}{dz^2} + \frac{1}{2} T_{\phi_\ast}(z) u(z)= 0,
\end{equation}
 Poincar\'e showed \footnote{Recall that in fact  Poincar\'e considered  the case $\chi_k=2$ for all $k$.} that the ratio $f=u_1/u_2$ of two independent solutions $u_1,u_2$ solves the uniformization problem in the sense that the metric $e^{\phi_\ast(z)}|\dd z|^2$ is the pull-back of the hyperbolic metric on the unit disk by $f$, i.e. the following holds
\begin{equation}\label{pullback}
e^{\phi_\ast(z)} = \frac{4 |f'(z)|^2}{ \Lambda \pi (1-|f(z)|^2)^2}.
\end{equation}
In particular, if one normalizes $u_1,u_2$ to have Wronskian $w= u_1'u_2-u_1u_2'$ equal to $1$ then $e^{- \phi_\ast(z)/2}= \sqrt{\Lambda/8}(|u_2|^2-|u_1^2|)$. The factor $e^{- \phi_\ast (z)/2}$ thus solves the following PDE version of the Fuchsian equation\footnote{ The fact that $e^{- \phi_\ast (z)/2}$ solves the PDE version of the Fuchsian equation can also be seen by definition of $T_{\phi_\ast  }$.}
\begin{equation}\label{FuchsPDE}
\partial^2_{zz} (e^{- \phi_\ast (z)/2}) +\frac{1}{2}  T_{\phi_\ast}(z) e^{- \phi_\ast (z)/2}= 0.
\end{equation}    
Therefore, equations \eqref{Fuchs} and \eqref{pullback} provide a link between constant curvature metrics and the uniformization problem of Riemann surfaces. Finally, let us mention that Poincar\'e left open the problem of characterizing the   complex numbers $c_k$ in \eqref{merostress} in terms of $\phi_\ast$. 

\medskip

\noindent More than eighty years later, Polyakov and Zamolodchikov suggested\footnote{See Takhtajan's lecture notes \cite{Tak}.} the following identity for the parameters $c_k$
\begin{equation}\label{Polyacc}
c_k= -\frac{1}{2}\partial_{z_k}  S_{(\chi_k,z_k)}(\phi_\ast)\footnote{Let us stress here that the function $\phi_\ast$ also depends on the $(\chi_k,z_k)$ though for notational simplicity we keep this dependence implicit.}, 
\end{equation}
where the function $\phi\in \Theta_{(\chi_k,z_k)}\mapsto S_{(\chi_k,z_k)}(\phi)$ is a functional, called the (classical) Liouville action (with conical singularities),   defined on some functional space   $\Theta_{(\chi_k,z_k)}$ (a space of functions with logarithmic singularities of the form \eqref{lcfasympt}) and that  must be  formally understood as 
\begin{equation}\label{LAformal}
 S_{(\chi_k,z_k)}(\phi)=\frac{1}{4\pi}\int_\C (|\nabla_z\phi(z)|^2+4\pi \Lambda e^{\phi(z)}) \dd^2z-\sum_{k=1}^n\chi_k\phi(z_k) .
\end{equation}
Let us mention here that relation  \eqref{Polyacc} has already been proved in a rather simple way by Takhtajan-Zograf \cite{Takbis} based on geometrical considerations. Yet,  expression \eqref{LAformal} is ill-defined since the functions in $\Theta_{(\chi_k,z_k)}$ have logarithmic singularities at $z_k$ (just like the function $\phi_\ast$). In order to give the precise definitions of $\Theta_{(\chi_k,z_k)}$ and $S_{(\chi_k,z_k)}$, we introduce the round metric $g(z) |dz|^2$ on the Riemann sphere where $g(z)$ is given by
\begin{equation*}
g(z)=\frac{4}{(1+\bar zz)^2}
\end{equation*}
with associated Green kernel $G$ with vanishing mean on the sphere, i.e. $\int_\C G(x,.) g(x)\dd^2x=0$ (where $\dd^2x$ denotes the standard Lebesgue measure). If we consider the standard Sobolev space 
\begin{equation*}
H^1(\hat{\C})= \left\lbrace  h \; : \; \int_\C [|\nabla_z h(z)|^2+ |h(z)|^2 g(z) ] \dd^2z< \infty \right\rbrace,
\end{equation*}
where the norm is defined by $|h|_{H^1(\hat{\C})  }= \left ( \int_\C (|\nabla_z h(z)|^2+ |h(z)|^2 g(z) )\dd^2z \right ) ^{1/2}$, then $\Theta_{(\chi_k,z_k)}$ is given by 
\begin{equation}\label{defTheta}
\Theta_{(\chi_k,z_k)} := \lbrace  \phi= h+ \ln g+ \sum_{k=1}^n \chi_k G(z_k,.); h \in  H^1(\hat{\C})  \rbrace.
\end{equation}
We endow $\Theta_{(\chi_k,z_k)}$ with the metric space structure   induced by the $H^{1}$-norm $$ d_{\Theta}(\phi_1,\phi_2):= \|\phi_1-\phi_2 \|_{H^1(\bbC)}.$$ 
Now, for $\phi \in \Theta_{(\chi_k,z_k)}$, the action $S_{(\chi_k,z_k)}(\phi)$ is defined by a  limiting procedure where one applies a regularization procedure around the points $z_k$ with logarithmic singularity  and  add diverging counter terms. More precisely, $S_{(\chi_k,z_k)}(\phi):=\lim_{\epsilon\to 0}S_\epsilon(\phi)$ where 
\begin{equation}  \label{procedlimit}
  S_\epsilon (\phi)  
 := \frac{1}{4\pi}\left[\int_{\C_\epsilon}  ( |\nabla_z \phi(z) |^2 +4\pi \Lambda e^{\phi(z)}) \dd^2z + R(\gep,\phi) \right]  
 \end{equation}
 and 
 \begin{equation}
  R(\gep,\phi):=- 2i \sum_{k=1}^n  \chi_k \oint_{|z-z_k|=\epsilon}   \phi(z) \frac{\overline{\dd z}}{\bar{z}-\bar{z_k}}+ 8 i \oint_{|z|=\frac{1}{\epsilon}}   \phi(z) \frac{\overline{\dd z}}{\bar{z}} +  2\pi \sum_{k=1}^n \chi_k^2 \ln \frac{1}{\epsilon}+ 32 \pi \ln \frac{1}{\epsilon},
\end{equation}
where the integration domain in the first integral is defined by $\C_\epsilon:=\C \setminus \cup_{k=1}^n B(z_k,\epsilon) \cup \lbrace |z| > \frac{1}{\epsilon}\rbrace$ and  the contour integrals $\oint$ are oriented counterclockwise. 

The purpose of this paper is to relate via a semiclassical analysis the classical Liouville action  $S_{(\chi_k,z_k)}$ to the recent rigorous probabilistic construction of Liouville Conformal Field Theory (LCFT hereafter)  given by David-Kupianen and the last two authors in \cite{DKRV}. Recall that the construction of \cite{DKRV} is based on the Gaussian Free Field (GFF). Moreover, the local conformal structure of LCFT was studied in \cite{KRV} (Ward and BPZ identities), paving the way to a proof of the DOZZ formula for the three point correlation function \cite{KRV1}. A byproduct of the semiclassical analysis is a  new proof of relation \eqref{Polyacc} in the spirit of the way Polyakov and Zamolodchikov discovered it via non rigorous asymptotic expansions on path integrals. We believe that the connection between LCFT (in a probabilistic setting) and the classical Liouville action  \eqref{procedlimit} is interesting per se as it is not straightforward (see Takhtajan's lecture notes and discussion \cite{Tak, Tak1}).

The construction of LCFT is based on the {\it quantum} Liouville action, the quantum analog of  \eqref{LAformal} (recall that \eqref{LAformal} is a formal definition and the exact definition requires a regularization procedure). In order to introduce the quantum Liouville action, let us set the change of variable $\phi=\varphi+ \ln g $ in the formal definition of the classical action \eqref{LAformal}; we get up to a global constant that
\begin{align*}
 S'_{(\chi_k,z_k)}(\varphi) & : = S_{(\chi_k,z_k)}(\varphi+\ln g)\\& = \frac{1}{4\pi}\int (|\nabla_z\varphi(z)+\nabla_z \ln g(z) |^2+4\pi \Lambda e^{\varphi(z)} g(z)) \dd^2z-\sum_{k=1}^n\chi_k (\varphi(z_k)+\ln g(z_k))     \\
 &   =  \frac{1}{4\pi}\int_{\C}\big( |\nabla_z\varphi (z) |^2+ 2R_{g}(z) \varphi(z)  g(z) +4 \pi \Lambda e^{ \varphi(z) } g(z)\big)\dd^2 z -\sum_{k=1}^n\chi_k(\varphi(z_k)+ \ln g(z_k))    \\
\end{align*}
where in the last line we disregard a global constant (independent of the $\chi_k,z_k$) and we introduced the Ricci curvature of the round metric $R_{g}(z):=-\frac{1}{g(z)}\Delta_z \ln g(z)$ (which in the particular case of the Riemann sphere is constant and equal to $2$).
Let us perform for $\gamma>0$ the change of variables $\alpha_k= \frac{\chi_k}{\gamma}$, $\mu= \frac{\Lambda}{\gamma^2} $ and work with $\gamma \varphi$ instead of $\varphi$ in which case we get
\begin{align}
S'_{(\chi_k,z_k),\gamma}(  \varphi)  &: = \frac{1}{\gamma^2}S'_{(\chi_k,z_k)}( \gamma \varphi) \label{finaldefLioucla}  \\
   &  =  \frac{1}{4\pi  }\int_{\C}\big(  |\nabla_z  \varphi (z) |^2+ \frac{2}{\gamma} R_{g}(z)  \varphi(z)   g(z)+ 4\pi \mu e^{ \gamma \varphi(z) }  g(z)\big)\dd^2 z -\sum_{k=1}^n\alpha_k (\varphi(z_k)+\frac{1}{\gamma} \ln g(z_k))    \nonumber
\end{align}
The construction of LCFT is based on the ``quantization" of the above action. Following general principles in quantum field theory, the theory should correspond to constructing the measure $e^{- S'_{(\chi_k,z_k),\gamma}( \varphi)} D\varphi$  where $D\varphi$ is the ill-defined ``Lebesgue measure" on the space of functions\footnote{It is a well known fact that the Lebesgue measure does not exist on infinite dimensional spaces.}. Following a standard procedure in constructive field theory (see Barry Simon's book \cite{BarrySimon}), the probabilistic construction of LCFT is based on interpreting the measure $e^{-  \frac{1}{4 \pi  }\int_{\C}  |\nabla_z  \varphi (z) |^2 \dd^2z}   D \varphi$ as the Gaussian Free Field (GFF)\footnote{There is an important subtelty here in the interpretation of this quadratic term; indeed, one must not forget to incorporate the spatial average $\int_{\C} \varphi(z) g(z) \dd^2z$ of the field $\varphi$ (with respect to $g$) in the definition of the GFF measure. The correct measure on this average is the standard Lebesgue measure on $\R$: in physics, this average is called the zero mode.} measure and then expressing the other terms of the action as functionals of the GFF along the following factorization:
\begin{equation*}
e^{- S'_{(\chi_k,z_k),\gamma}(  \varphi)} D\varphi= \left ( e^{ \sum_k\alpha_k (\varphi(z_k)+\frac{1}{\gamma} \ln g(z_k))   } e^{-  \frac{1}{4\pi  }\int_{\C} (  \frac{2}{\gamma} R_{g}(z)  \varphi(z)  g(z) +4 \pi \mu e^{ \gamma \varphi(z) } g(z) )\dd^2 z}     \right )  e^{-  \frac{1}{4 \pi  }\int_{\C}  |\nabla_z  \varphi (z) |^2 \dd^2z}   D \varphi.
\end{equation*}   
However, the GFF is not a function but rather a Schwartz distribution hence the exponential term $\int_\C e^{ \gamma \varphi(z) } g(z) \dd^2 z$ is ill-defined. In order to make sense of the this term, a renormalization procedure is required and, in order to preserve the conformal invariance properties at the quantum level, one must modify the above action by replacing the classical value $\frac{2}{\gamma}$ in front of  $R_{g}(z)  \varphi(z)  g(z)$ by the \emph{quantum} value $Q:= \frac{2}{\gamma}+\frac{\gamma}{2}$ (see \cite{DKRV}). Of course, the extra correction term $\frac{\gamma}{2}$ vanishes when one considers the semiclassical regime $\gamma \to 0$. This leads to the following formal definition of the quantum Liouville action (where one removes the logarithmic singularities)
\begin{equation}\label{actionLiouvilleQ}
S(\varphi,g):= \frac{1}{4 \pi}\int_{\C}\big(|\nabla_z\varphi (z) |^2+ Q R_{g}(z) \varphi(z)  g(z) +4 \pi \mu e^{\gamma \varphi(z) }  g(z)\big)\dd^2 z   .
\end{equation}
where $\gamma$ is a positive parameter belonging to  $(0,2)$, $Q=\frac{\gamma}{2}+\frac{2}{\gamma}$ and $\mu>0$ is a positive parameter called the cosmological constant. As we will see shortly, the quantum Liouville measure $e^{- S(\varphi,g)} D\varphi$ is in fact defined on the dual space $H^{-1}(\hat\C)$ of  $H^{1}(\hat\C)$, which is defined as the completion of the set of smooth functions on $\hat \C$ with respect to the following norm for $f$ smooth
\begin{equation*}
|f|_{H^{-1}(\hat\C)}=  \sup_{h \in H^{1}(\hat \C), \: |h|_{H^1(\hat{\C})  } \leq 1} \left |    \int_\C    f(x) h(x)  g(x) \dd^2x  \right |  . 
\end{equation*}
The quantum Liouville theory is then defined by its functional expectation (called path integal in the physics literature)
\begin{equation}\label{Liouvmeasintro}
\langle F\rangle_{\gamma,\mu} =\int_{H^{-1}(\hat\C)}F(\phi)e^{-S(\varphi,g)}D\varphi
\end{equation}
for every continuous function $F$ on $H^{-1}(\hat\C)$ and where $\phi= \varphi+\frac{Q}{2} \ln g$ is the {\bf Liouville field}. The main observables in LCFT 
are the correlations of the fields $V_\alpha(z)= e^{\alpha \phi (z)}$ under the measure \eqref{Liouvmeasintro}. For $\alpha_1, \cdots, \alpha_n$  and $F$ any  bounded measurable functional, we set 
\begin{equation*}
\langle F(\phi)  \prod_{k=1}^n V_{\alpha_k}(z_k)   \rangle_{\gamma,\mu}:=  \int_{H^{-1}(\hat\C)} F(\phi)   \prod_{k=1}^n V_{\alpha_k}(z_k)    e^{-S(\phi,g)}D\phi .
\end{equation*}
We will recall the rigorous probabilistic definition in the next subsection; the case $F=1$ corresponds to the correlations. Let us just mention that inserting quantities like $V_{\alpha_k}(z_k)$ correspond at the quantum level to adding a logarithmic singularity to $\phi$ at the point $z_k$ and with weight $\alpha_k$ (this a consequence of the classical Girsanov theorem of probability theory); this property is of course to be expected from the previous discussion at the classical level.

Now, gathering the above considerations, the semiclassical regime is the limit of LCFT when $\gamma$ goes to $0$ with $\alpha_k=\frac{\chi_k}{\gamma}$ and $\mu= \frac{\Lambda}{\gamma^2}$ for fixed $\Lambda$ and $\chi_k$; it is natural to expect that the following semi-classical limit holds when $F$ is a continuous function on $H^{-1} ( \hat\C)$  
\begin{equation}\label{Liouvmeasintrosemicl}
\langle F(\gamma \phi)  \prod_{k=1}^n V_{\alpha_k}(z_k)   \rangle_{\gamma,\mu}\underset{\gamma \to 0}{\sim} C e^{- \frac{ S_{(\chi_k,z_k)}(\phi_\ast)}{\gamma^2}} F(\phi_\ast)
\end{equation}
where $\phi_\ast$ solves the Liouville Equation  \eqref{Liouv} with logarithmic singularities \eqref{lcfasympt} and $C>0$ is some constant. One of the main results of this paper is to show that this is indeed the case (see Proposition \ref{asymppart} below). Once statement \eqref{Liouvmeasintrosemicl} has been established rigorously, following the idea of Polyakov and Zamolodchikov, it is natural to exploit the above semiclassical limit to recover the accessory parameters by using the BPZ equations\footnote{The argument  of Polyakov and Zamolodchikov was in fact based on the stress-energy tensor of LCFT but this is a minor point.}. The BPZ differential equations are the quantum analogues of \eqref{FuchsPDE}; it was shown in \cite{KRV} that the following BPZ differential equation holds for the field $V_{- \frac{\gamma}{2}}(z)$
\begin{align}
  \frac{4}{\gamma^2}\partial_{zz}^2\langle   V_{- \frac{\gamma}{2}}(z) \prod_{l=1}^n V_{\alpha_l}(z_l)   \rangle_{\gamma,\mu}   & + \sum_{k=1}^n \frac{\Delta_{\alpha_k}}{(z-z_k)^2}  \langle  V_{- \frac{\gamma}{2}}(z) \prod_{l=1}^n  V_{\alpha_l}(z_l)  \rangle_{\gamma,\mu}   \nonumber \\ &   + \sum_{k=1}^n \frac{1}{z-z_k}  \partial_{z_k} \langle  V_{- \frac{\gamma}{2}}(z) \prod_{l=1}^n  V_{\alpha_l}(z_l)  \rangle_{\gamma,\mu}    
  =  0  , \label{BPZ}
\end{align}
 where $\Delta_\alpha= \frac{\alpha}{2}  (Q-\frac{\alpha}{2})$ is called the conformal weight of $V_\alpha(z)$ (while $\Delta$ is also used for the Laplacians, our use of this notation should not yield any ambiguity). Since the BPZ differential equations are the quantum analogues of the classical equation \eqref{FuchsPDE}, one should recover \eqref{FuchsPDE} by taking the semiclassical limit $\gamma \to 0$ (recall that $\alpha_k=\frac{\chi_k}{\gamma}$ and $\mu= \frac{\Lambda}{\gamma^2}$ for fixed $\Lambda$ and $\chi_k$). Indeed, exploiting \eqref{Liouvmeasintrosemicl} and the convergence $\gamma^2 \Delta_{\alpha_k} \underset{\gamma \to 0}{\rightarrow} \chi_k-\chi_k^2/4$ one gets asymptotically for small $\gamma$
\begin{align*} 
  0  & = \partial_{zz}^2\langle   V_{- \frac{\gamma}{2}}(z) \prod_{l=1}^n V_{\alpha_l}(z_l)   \rangle_{\gamma,\mu}    +  \frac{\gamma^2}{4} \sum_{k=1}^n \frac{\Delta_{\alpha_k}}{(z-z_k)^2}  \langle  V_{- \frac{\gamma}{2}}(z) \prod_{l=1}^n  V_{\alpha_l}(z_l)  \rangle_{\gamma,\mu}   \\
  &  \quad \quad \quad \quad \quad \quad \quad \quad \quad \quad \quad \quad \quad \quad \quad \quad  \quad \quad \quad \quad \quad + \frac{\gamma^2}{4}  \sum_{k=1}^n \frac{1}{z-z_k}  \partial_{z_k} \langle  V_{- \frac{\gamma}{2}}(z) \prod_{l=1}^n  V_{\alpha_l}(z_l)  \rangle_{\gamma,\mu}     \\
 & 
= \partial_{zz}^2  \left (   e^{-\frac{\phi_\ast}{2}(z)}  e^{- \frac{ S_{(\chi_l,z_l)}(\phi_\ast)}{\gamma^2}} \right )   + \sum_{k=1}^n \frac{\chi_k/4-\chi_k^2/16}{(z-z_k)^2}      e^{-\frac{\phi_\ast}{2}(z)}  e^{- \frac{ S_{(\chi_l,z_l)}(\phi_\ast)}{\gamma^2}} \\
&\quad \quad \quad \quad \quad \quad \quad \quad \quad \quad \quad \quad \quad \quad \quad \quad  \quad \quad \quad \quad \quad+\frac{\gamma^2}{4}  \sum_{k=1}^n \frac{1}{z-z_k}  \partial_{z_k} \left (  e^{-\frac{\phi_\ast}{2}(z)} e^{- \frac{ S_{(\chi_l,z_l)}(\phi_\ast)}{\gamma^2}} \right )   +o(1)\\
& = e^{- \frac{ S_{(\chi_l,z_l)}(\phi_\ast)}{\gamma^2}} \left ( \partial_{zz}^2    e^{-\frac{\phi_\ast}{2}(z)}      + \sum_{k=1}^n \frac{\chi_k/4-\chi_k^2/16}{(z-z_k)^2}      e^{-\frac{\phi_\ast}{2}(z)} + \frac{1}{4} \sum_{k=1}^n \frac{ \partial_{z_k} S_{(\chi_l,z_l)}(\phi_\ast)  }{z-z_k} e^{-\frac{\phi_\ast}{2}(z)}   \right )   +o(1).
\end{align*}
This leads to the desired relation \eqref{Polyacc}: this heuristic derivation can be made rigorous and is the content of Corollary \ref{CorPolyakov}. The main reason why the above derivation is not an immediate consequence of  \eqref{Liouvmeasintrosemicl} is for regularity reasons; more specifically, one must justify that one can differentiate equivalence \eqref{Liouvmeasintrosemicl}. 

  Now, we proceed with the statement of the main results of this paper as well as related open problems. In order to do so, we first recall the probabilistic definition of LCFT.

\section{Main results}

\subsection{Background and notations}\label{sec:backgr}
 
In this section, we recall the precise definition of the Liouville action and LCFT as given in  \cite{DKRV}.

\subsubsection*{Convention and notations.}

In what follows, in addition to the complex variable $z$, we will also consider variables $x,y$ in $\C$ and for integer $n \geq 3$ variables  $z_1, \cdots, z_n$ which also belong to $\C$. 

The variables $x,y$ (and sometimes $z$) will typically  be variables of integration: we will denote by $\dd^2x$ and $\dd^2y$ (and $\dd^2z$) the corresponding Lebesgue measure on $\C$ (seen as $\R^2$). We will also denote $|\cdot|$ the norm in $\C$ of the standard Euclidean (flat) metric and for all $r>0$ we will denote by $B(x,r)$ the Euclidean ball of center $x$ and radius $r$. 

\subsubsection*{LCFT on  $\hat{\C}$}  
 
  To define the measure  \eqref{Liouvmeasintro} it is natural to start with the quadratic part of the action functional  \eqref{actionLiouvilleQ}  which naturally gives rise to a Gaussian measure, the Gaussian Free Field (GFF) (we refer to \cite[Section 4]{dubedat} or \cite{She07} for an introduction to the topic). As is well known the GFF on the plane is defined modulo a constant but in LCFT this constant has to be included as an integration variable in the measure \eqref{Liouvmeasintro}. The way to proceed is to replace $\varphi$ in \eqref{actionLiouvilleQ}  by $c+X$  where $c\in\R$  is integrated w.r.t to Lebesgue measure  and $X$ is the Gaussian Free Field on $\C$ centered with respect to the round metric, i.e. which satisfies $\int_\C X(x) g(x)\dd^2x=0$ for $g(x)= \frac{4}{(1+|x|^2)^2}$. The covariance of $X$\footnote{The field $X$ was denoted $X_{g}$ in the article \cite{DKRV} or the lecture notes \cite{RV}.} is given explicitly for $x,y \in \C$ by 
\begin{equation}\label{hatGformula}
\E [ X(x)X(y)] = G(x,y)=\ln\frac{1}{|x-y|}-\frac{1}{4}(\ln g(x)+\ln  g(y))+\kappa
\end{equation}
where $\kappa:=\ln 2-\frac{1}{2}$.

\subsubsection*{Gaussian multiplicative chaos}
The field $X$ is distribution valued and to define its exponential a renormalization procedure is needed. We will work with a mollified regularization of the GFF, namely $X_{\epsilon}=\rho_\epsilon \ast  X   $ with $\rho_\epsilon(x)= \frac{1}{\epsilon^2}\rho (\frac{|x|^2}{\epsilon^2})$ where $\rho$ is $C^{\infty}$ non-negative with compact support in $[0,\infty[$ and such that $\pi \int_0^\infty \rho(t) \dd t=1$.  The variance of  $X_{\epsilon}(x)$ satisfies
 \begin{equation}\label{circlegreen}
\lim_{\epsilon\to 0} \: (\E[X_{\epsilon} (x)^2]+\ln (a\epsilon))
=-\frac{1}{2}\ln g(x)
\end{equation} 
uniformly on $\C$ where the constant $a$ depends on the regularization function $\rho$. 
Define the measure 
\begin{equation}\label{Meps1}
M_{\gamma,\epsilon}(\dd^2x):= e^{\frac{\gamma^2}{2}\kappa}
:e^{\gamma X_{\epsilon}(x)}: g(x) \dd^2x.
\end{equation}
where we have used the Wick notation for a centered Gaussian random variable
$:e^Z:= e^{Z-\frac{1}{2}\bbE [Z^2]}$ (see Section \ref{wick}).
While the factor  $e^{\frac{\gamma^2}{2}\kappa}$ plays no role, this 
normalization of $M_{\gamma,\epsilon}(\dd^2x)$ has been chosen to match the standards of the physics literature.
For $\gamma\in [0,2)$,   this sequence of measures converges  in probability  in the sense of weak convergence of measures
\begin{equation}\label{law}
M_{\gamma}=\lim_{\epsilon\to 0}M_{\gamma,\epsilon}.
 \end{equation}
This limiting measure is non trivial and is (up to the multiplicative constant $  e^{\frac{\gamma^2}{2}\kappa} $) Gaussian multiplicative chaos (GMC) of the field $X$ with respect to the measure $ g(x)\dd^2 x$ (see Berestycki's paper \cite{Ber} for an elementary approach and references). 

 \subsubsection*{Liouville measure } The Liouville measure $e^{-S(\varphi,g)}D\varphi$ with $S$ given by \eqref{actionLiouvilleQ}  is now defined as follows. Since $R_{ g}=2$ and $\int_\C X(x)  g(x) \dd^2x=0$ the linear term becomes
$
\int_\C(c+X(x))  g (x) \dd^2x=4\pi c
$.  This leads to the
definition (recall that $\mu>0$ is a fixed parameter)
\begin{equation}\label{Liouvillemeasure}
 \nu(dX,dc):=e^{-2Q c}     e^{- \mu e^{\gamma c} M_{\gamma}(\C)    }\P(\dd X)   \: \dd c.
\end{equation}
Here $\P(\dd X)$ denotes the Gaussian Free Field probability measure on $H^{-1}(\hat \C)$. Note that the random variable $M_{\gamma}(\C) $ is almost surely finite and that  $\E [M_{\gamma}(\C)] =e^{\frac{\gamma^2}{2}\kappa}\int_\C   g(x) \dd^2x<\infty
$. This implies that the total mass of the measure $\nu$ is  infinite  since for $M_{\gamma}(\C) <\infty$ the $c$-integral diverges at $-\infty$. While we have formally defined $\nu$ as a measure on $H^{-1}(\hat \C)\times \bbR$, we are solely interested in the behavior of the  \textit{Liouville field}, defined by
\begin{equation}\label{Liouvillefield}
\phi:=X+\frac{Q}{2}\ln g+c.  
\end{equation}
Note that by \eqref{circlegreen}, the term appearing in the exponential in \eqref{Liouvillemeasure} can be expressed as 
 \begin{equation}\label{Meps}
e^{c\gamma} M_{\gamma}(d^2x)=\lim_{\epsilon\to 0}(A\epsilon)^{\frac{\gamma^2}{2}}e^{\gamma \phi_\epsilon(x)}  \dd^2x
\end{equation}
for $A=ae^\kappa$ and $\phi_\epsilon:=\rho_\epsilon\ast\phi$ being the smoothened version of $\phi$. 
We denote 
averages with respect to $\nu$ by $\langle\,\cdot\,\rangle_{\gamma, \mu}$

\subsubsection*{Liouville correlation functions.} The vertex operators $V_\alpha(z)=e^{\alpha \phi(z)}$  need to be regularized as well.
Equation \eqref{Meps1} and \eqref{Meps} provide several ways of regularizing which end up to be equivalent when considering limits. 
For vertex operators, we introduce for $\alpha\in\R$ and $z\in\C$
\begin{equation}\label{Vdefi}
V_{\alpha, \epsilon}(z):=e^{\alpha c} e^{\frac{\alpha^2 \kappa}{2}}:e^{\alpha X_{\epsilon}(z) }: g (z) ^{\Delta_\alpha}.
\end{equation}
where recall that $\Delta_\alpha= \frac{\alpha}{2} (Q-\frac{\alpha}{2})$. Let us fix $z_1,\dots, z_n$, $n$ distincts points in $\bbC$ and associated
 weights $\alpha_1,\dots,\alpha_n$ in $\bbR$. Here and below we use the notation $\langle\,\cdot\,\rangle_{\gamma, \mu,\epsilon} $ for the regularized Liouville measure where in \eqref{Liouvillemeasure} we replace  $M_{\gamma}(\dd^2 x)$ by $M_{\gamma, \gep}(\dd^2 x)$.  
Now, it was shown in \cite{DKRV} that the  limit
\begin{equation}\label{lcorre}
 \langle \prod_{k=1}^n  V_{\alpha_k}(z_k) \rangle_{\gamma, \mu} :=  4 e^{-2\kappa Q^2}\: \underset{\epsilon \to 0}{\lim}  \:  \langle \prod_{k=1}^n  V_{\alpha_k,\epsilon}(z_k)  \rangle_{\gamma, \mu,\epsilon}\footnote{The global constant $4 e^{-2\kappa Q^2}$ which depends on $\gamma$ plays no role but it is included to match with the standard  physics literature which is based on the celebrated DOZZ formula. This constant was not included in the definitions in \cite{DKRV} or \cite{RV}.}
\end{equation}
exists and is finite if and only if $ \sum_{k=1}^n\alpha_k>2Q$.  Moreover, under this condition, the limit is non zero  if and only if $\alpha_k<Q$ for all $k$. The conditions 
\begin{equation}\label{TheSeibergbounds}
 \sum_{k=1}^n\alpha_k>2Q, \quad \quad \forall k, \: \alpha_k<Q
\end{equation}
are the Seiberg bounds  originally introduced in \cite{seiberg}. These two conditions are the quantum equivalent of those presented in \eqref{Negcurv}. This is really transparent in the semiclassical regime  where we fix $\alpha_k= \frac{\chi_k}{\gamma}$ for fixed $\chi_k$'s and $\gamma$ goes to $0$; in the limit, one gets $\chi_k  \leq 2$ and $ \sum_{k=1}^n \chi_k  \geq 4$. With the exception of Section \ref{ops} on open problems, we assume throughout this paper that the Seiberg bounds  \eqref{TheSeibergbounds} are satisfied and (without entailing any restriction since $V_0(z)=1$) that $\alpha_k \neq 0$ for all $k$.  Note that this also implies that $n \geq 3$.  

%

 \subsubsection*{Reduction to Multiplicative Chaos.} In order to keep this work as self contained as possible, we remind the basics of the construction of the Liouville correlations. The main idea is that one can express these correlations as functions of GMC measures with log singularities. As a first step using the explicit expression  \eqref{Vdefi},
 and the change of variable $a=\gamma c+\log M_{\gamma, \gep}(\bbC) $
 we can factorize the respective role of $X$ and the constant. Setting 
 \begin{equation}\label{defs}
 s:=\frac{\sum_{k=1}^n \alpha_k-2Q}{\gamma}
 \end{equation}
this yields (see  \cite{DKRV} for more details):
  \begin{multline}\label{lachange}
   \langle \prod_{k=1}^n  V_{\alpha_k}(z_k)  \rangle_{\gamma,\mu} = 4 \: e^{-2\kappa Q^2}\:  \lim_{\epsilon\to 0} \int_{\R}   e^{-2Q c}  \:   \E \left [  \prod_{k=1}^n  V_{\alpha_k,\epsilon}(z_k)        e^{- \mu e^{\gamma c}  M_{\gamma, \epsilon} (\bbC)    } \right ]   \dd c\\
 =   4 \: e^{\frac{\kappa}{2}\sum_{k=1}^n\alpha_k^2-2\kappa Q^2}\gamma^{-1}\left(\int_{\bbR} e^{-\mu e^ a} e^{as} \dd a\right) \lim_{\epsilon\to 0}\E  \left [  \prod_{k=1}^n :e^{{\alpha_k} X_\epsilon(z_k)}:   g (z_k)^{\Delta_{\alpha_k}}    \left ( M_{\gamma, \epsilon} (\bbC)   \right )   ^{-s}    \right ].
\end{multline}
Since  \eqref{TheSeibergbounds} ensures that $s>0$, making a  change of variables $A=\mu e^a$ in the integral $\int_{\bbR} e^{-\mu e^ a} e^{as} \dd a$ turns it into $\mu^{-s}\Gamma(s)$  where $\Gamma$ is Euler's $\Gamma$ function: $\Gamma(s)=\int_0^\infty A^{s-1} e^{-A} \dd A$.
Now  using Girsanov's theorem  (see  \cite{DKRV}) we may trade $:e^{{\alpha_k} X_\epsilon(z_k)}:$ in the expectation for a shift of $X$.
Then setting 
\begin{equation}\label{defz0}
Z_0:= \int_{\C}  e^{\gamma \sum_{k=1}^n \alpha_k G(z_k,x) -\frac{\gamma^2\kappa}{2}}  M_\gamma (d^2x)
=\lim_{\epsilon \to 0} \int_{\C}  e^{\gamma \sum_{k=1}^n \alpha_k G(z_k,x)} :e^{\gamma X_{\epsilon} (x)}:   g(x) \dd^2 x ,
\end{equation}
we end up with the expression
\begin{equation}\label{funamental}
 \langle    \prod_{k=1}^n V_{\alpha_k}(z_k)  \rangle_{\gamma,\mu}   =  K({\bf z}) \mu^{-s}\Gamma(s)    \E \left [  Z_0^{-s}  \right ]
\end{equation}
where for ${\bf z}=(z_1, \dots, z_n)$
\begin{equation}\label{Kdef}
K({\bf z}):=4 \gamma^{-1} e^{-\frac{s\kappa \gamma^2}{2}}\left ( \prod_{k=1}^n g(z_k)^{\Delta_{\alpha_k}}  \right ) e^{ \frac{1}{2}\sum_{k \not =j}\alpha_k\alpha_{j}G(z_k,z_{j})+\frac{\kappa}{2} \sum_{k=1}^n \alpha_k^2-2\kappa Q^2}.
\end{equation}
Thus, the Liouville correlations can be expressed in terms of the negative moments of the random variable $Z_0$. In particular, the Seiberg bounds $\alpha_k<Q$ for all $k$ are the condition of integrability of $e^{\gamma \sum_{k=1}^n \alpha_k G(z_k,x)}$ against the chaos measure  $M_\gamma(\dd^2x)$ (see  \cite{DKRV}). 
We recall the following result on the BPZ equations proved in \cite{KRV}:
\begin{theorem}[Theorem 2.2 in \cite{KRV}]\label{BPZTH} Suppose $ -\frac{\gamma}{2}+\sum_{k=1}^n\alpha_k>2Q$. Then the BPZ equation \eqref{BPZ} holds 
 in $\mathbb{C} \setminus  \lbrace  z_1, \cdots, z_n \rbrace$.
\end{theorem}


\noindent Similarly  to the correlation  $\langle \prod_{k=1}^n  V_{\alpha_k}(z_k)  \rangle_{\gamma,\mu}$, we can also define a limit for the sequence of probability measures induced by the multiplication by vertex operators. Setting
\begin{equation}\label{vertexcorr3}
\P_{\mu, (\alpha_k, z_k),\epsilon}:=   \frac{1}{   \langle \prod_{k=1}^n  V_{\alpha_k,\gep}(z_k)  \rangle_{\gamma,\mu,\epsilon} } \prod_{k=1}^n V_{\alpha_k,\epsilon}(z_k)\, \nu (\dd X,\dd c).
\end{equation}
we have (in the topology of weak convergence of measures on  $H^{-1}(\bbC)\times \bbR$)
\begin{equation}\label{vertexcorr4}
\lim_{\epsilon\to 0}\P_{\mu, (\alpha_k, z_k),\epsilon}=\P_{\mu,(\alpha_k, z_k)}.
\end{equation}
Under the probability measure $\P_{\mu,(\alpha_k, z_k)}$ (with expectation denoted by $\E_{\mu,(\alpha_k, z_k)}$), the distribution of the Liouville field \eqref{Liouvillefield} is given by (for any continuous bounded function $F: H^{-1}(\bbC) \to \bbR$)
\begin{multline}
 \E_{\mu,(\alpha_k, z_k)} [F(\phi)] 
\\ = 
\int_\bbR  \frac{ \E\left[  F\left( X+\sum_{k=1}^n\alpha_k G(z_k,.) +\frac{Q}{2} \ln g  +  \frac{a-\ln Z_0}{\gamma}- \frac{\gamma\kappa}{2}  \right)  Z_0^{-s}    \right]}{\bbE[Z^{-s}_0]} \,    \frac{ \mu^s e^{a s} e^{-\mu e^a}}{\gG(s)} \dd a .\label{defLiouvillefield}
\end{multline}  
where recall that in this expression $s=\frac{\sum_{k=1}^n \alpha_k-2Q}{\gamma}$. This formula is obtained using the same change of variables performed in \eqref{lachange}.

\begin{remark}
 Note that we have formally defined  $\P_{\mu,(\alpha_k, z_k)}$ as a probability on $H^{-1}(\bbC) \times \R$ and  we have 
 \begin{equation}
 \E_{\mu,(\alpha_k, z_k)} [F(X,c)] 
 = \frac{ \gamma K({\bf z})}{\langle    \prod_{k=1}^n V_{\alpha_k}(z_k)  \rangle_{\gamma,\mu} }
\int_\bbR  \E\left[  F\left( X+\sum_{k=1}^n\alpha_k G(z_k,.), c \right) e^{-\mu e^{c\gamma +\frac{\gamma^2 \kappa}{2}}Z_0}  \right]\dd c 
 \end{equation}
but this expression is of lesser interest to us compared to  \eqref{defLiouvillefield} since only $\phi$ has a physical interpretation.

\end{remark}


%

\subsection{The semiclassical limit: statement of the main results}

In this section, we state the main results of the paper which are derived in the semiclassical regime. In this regime we
fix the values of $(\chi_k)_{k=1}^n$ satisfying \eqref{Negcurv} (recall that this implies $n \geq 3$)  and $\Lambda>0$ and 
set 
\begin{equation}\label{regime}
 \mu= \frac{\Lambda}{\gamma^2}  \quad \text{ and }  \quad  \forall k, \quad \alpha_k=\frac{\chi_k}{\gamma}
\end{equation}
and  we let $\gamma$ tend to $0$. 

\subsubsection*{The semiclassical limit of the correlation functions}

We introduce the following constant (depending on the $\chi_k$)
\begin{equation}\label{Theconstant}
C_{\star}((\chi_k))=2  \ln 4+ 2  \sum_{k=1}^n \chi_k  ( \frac{1}{2} \ln 2+\kappa)-\frac{1}{2} \kappa \sum_{k=1}^n \chi_k^2+\frac{1}{8 \pi } ( 4-\sum_{k=1}^n \chi_k)  \int_{\C}    g(z) \ln   g(z)  \dd^2z.
\end{equation}

Our  first achievement is to obtain sharp asymptotics for the correlation function (we direct the reader to Section \ref{wick} for details about the the Wick notation $:X^2:$).
Recalling that $\phi_*$ is the solution of \eqref{Liouv}-\eqref{lcfasympt} we set 
\begin{equation}\label{mustar}
\mu^*(\dd^2 x)=  \frac{\gL}{\sum_{k=1}^n \chi_k-4} e^{\phi_*(x)} \dd^2 x.
\end{equation}
Notice that $\mu^*$ is a probability measure on $\C$ thanks to \eqref{Gauss}.

\begin{proposition}\label{asymppart}
In the regime \eqref{regime}, we have
\begin{multline*}
 \langle \prod_{k=1}^n  V_{\alpha_k}(z_k)  \rangle_{\gamma,\frac{\Lambda}{\gamma^2} } \stackrel{\gamma \to 0}{\sim} 4 
 \sqrt{2\pi}  \frac{\Lambda}{(\sum_{k=1}^n \chi_k-4)^{3/2}} e^{\frac{\frac{\kappa}{2} \sum_{k=1}^n \chi_k^2-8 \kappa +C_{\star}((\chi_k))}{\gamma^2}} e^{-\frac{S_{(\chi_k,z_k)}(\phi_\ast)}{\gamma^2}}    \\
 \times e^{-\frac{\kappa}{2}(\sum_{k=1}^n \chi_k-4 )}   \left (e^{2(\ln 2 -1)} \, \int_\C e^{\phi_*(x)-\frac{1}{4\pi}\phi_*(x)} \dd^2 x   \right ) \E\left[   e^{   \frac{\sum_{k=1}^n \chi_k-4}{2}  \left[\left(\int_\C X(x) \mu^*(\dd^2x)   \right)^2 - \int_\C :X(x)^2: \mu^*(\dd^2x)   \right]     }      \right]   
\end{multline*}
\end{proposition}

\subsubsection*{Convergence of the Liouville field}

Recall the distribution   $\P_{\mu,(\alpha_k, z_k)}$ of the Liouville field defined in Equation \eqref{defLiouvillefield}.

We have the following result:

\begin{theorem}\label{thSemiclassical}
We have the following semi-classical approximation in the regime \eqref{regime} 

\begin{itemize}

\item [(i)]
The field $\gamma \phi$ converges in probability as $\gamma$ goes to $0$ towards $\phi_{\ast}$. 
\item[(ii)]
The field $\phi-\frac{1}{\gamma} \phi_{\ast}$ converges in distribution towards the field $Y+\hat{X}$ where $Y$ is a standard centered Gaussian of variance $\frac{1}{\sum_{k=1}^n\chi_k-4}$ and $\hat{X}$ is an independent massive Free Field with average zero in the background metric $ e^{\phi_{\ast}(z)} |\dd z|^2$.
More specifically $\hat{X}=X_m-\int_{\C} X_m(x) \mu^*(\dd^2x)$ where the distribution of $X_m$ is absolutely continuous with respect to that of the original GFF $X$ with density given by 
\begin{equation*}
  \frac{e^{   \frac{\sum_{k=1}^n \chi_k-4}{2}  \left[\left(\int_\C X(x) \mu^*(\dd^2x)   \right)^2 - \int_{\C} :X(x)^2: \mu^*(\dd^2x)\right]     } }{\bbE \left[e^{   \frac{\sum_{k=1}^n \chi_k-4}{2}  \left[\left(\int_\C X(x) \mu^*(\dd^2x)   \right)^2 - \int_{\C} :X(x)^2: \mu^*(\dd^2x)\right]}\right] }
\end{equation*}
\end{itemize}

\end{theorem}
 
\subsubsection*{Large deviations of the Liouville field}

Now, we can state the following large deviation principal for the field $\gamma \phi$:

\begin{proposition}\label{propLargedev}
In the regime given by \eqref{regime}, the field $\gamma \phi $ satisfies a large deviation principle on $\Theta_{(\chi_k,z_k)}$ with   rate $\gamma^{-2}$ and good rate function  $S_{(\chi_k,z_k)}$    given by the limiting procedure \eqref{procedlimit}, namely for every open subset of $\cO\subset\Theta(\chi_k,z_k)$ and closed subset $K\subset\Theta(\chi_k,z_k)$,
we have  
\begin{equation}
\begin{split}
\liminf_{\gamma\to 0} \gamma^{-2} \log \P_{\frac{\Lambda}{\gamma^2},(\alpha_k, z_k)}[ \gamma\phi \in \cO ]&\ge -\Big(\min_{\phi \in \cO } S_{(\chi_k,z_k)}(\phi)-S_{(\chi_k,z_k)}(\phi_*)\Big),\\
\limsup_{\gamma\to 0} \gamma^{-2} \log \P_{\frac{\Lambda}{\gamma^2},(\alpha_k, z_k)}[ \gamma\phi \in K ]&\le -\Big(\min_{\phi \in K } S_{(\chi_k,z_k)}(\phi)-S_{(\chi_k,z_k)}(\phi_*)\Big).
\end{split}
\end{equation}

\end{proposition}

\subsubsection*{A probabilistic proof of  the Takhtajan-Zograf theorem on the accessory parameters}

As a corollary of our techniques, we obtain a new (probabilistic) proof of relation \eqref{Polyacc} (previously proved in \cite{Takbis}):

\begin{corollary}\label{CorPolyakov}
The relation \eqref{Polyacc} holds.
\end{corollary}

\subsection{Open problems}\label{ops}

\subsubsection*{The semiclassical limit in the nonnegative curvature case}

The negative curvature condition $\sum_{k=1}^n \chi_k >4$ is necessary to properly define our Liouville measure. As can be seen in Equation \eqref{defLiouvillefield} the integral in $a$ diverges if the quantum analogue of this condition (namely $s=\frac{\sum_{k=1}^n \alpha_k-2Q}{\gamma}>0$) is not satisfied. However, we can bypass this obstruction by considering the measure conditioned on a fixed value of $a$ (without loss of generality we can consider $a=0$) 
\begin{equation}\label{UnitvolumeLiouville}
 \E_{\frac{\Lambda}{\gamma^2},(\alpha_k, z_k)}^{a=0}[F(\phi)]  =   \E\left[  F( X+\frac{1}{\gamma}\sum_k \alpha_k G(z_k,.) +Q/2 \ln g - \frac{\ln Z_0}{\gamma}  -\frac{\gamma \kappa}{2}) Z^{-s}_0 \right] /\E[ Z^{-s}_0].
\end{equation}   
Considering this expression is motivated by considering the law of the random measure on $\C$ defined by $e^{\gamma \phi(x)}\dd^2x$ under the probability law \eqref{defLiouvillefield}. It is then easy to check that the total mass of this random measure has law $\Gamma(s,\mu)$ (i.e. with density proportional to $\mu^{-s} A^{s-1} e^{-\mu A}$), and that the law of the Liouville field, conditionally on the total mass of $e^{\gamma \phi(x)} \dd^2x$ being $1$,    is described by \eqref{UnitvolumeLiouville}. In order to state a conjecture in the unit volume setting, we need to find a natural variational problem associated to it.  In this section, we assume the following conditions which were introduced by Troyanov \cite{troyanov} hold 
\begin{equation}\label{Troya}
\forall k, \; \chi_k <2, \quad \quad    4-\sum_{k=1}^n \chi_k < 4 \wedge \min_{k=1}^n  (4-2\chi_k).
\end{equation}
It is rather easy to check that condition \eqref{Troya} implies $n \geq 3$. The reason for the above condition will become clear shortly; let us just notice that condition \eqref{Troya} extends condition \eqref{Negcurv}.  Consider $\phi\in \Theta_{(\chi_k,z_k)}$ 
and its canonical decomposition
\begin{equation*}
 \phi= h+ \ln g+ \sum_{k} \chi_k G(z_k,.)
\end{equation*}
where $h  \in  H^1(\hat{\C})$. Normalizing $e^{\phi}$ to have unit volume amounts to replacing $\phi$ by 
\begin{equation}\label{unitvol}
\phi-\ln \int_{\C} e^{\phi(z)} \dd^2z=  h+ \ln g+ \sum_{k=1}^n \chi_k G(z_k,.)- \ln \left (  \int_\C e^{ h(z)+\sum_{k=1}^n \chi_k G(z_k,z)} g(z) \dd^2z  \right ).
\end{equation}
The $h$ satisfying such a relation is unique if we require $h$ to be of vanishing mean $ \int_\C h(x)g(x) \dd^2x=0$. Let us register this relation by defining the  injective  map $T$ on $ \bar H^1(\hat{\C}):=\{ h\in H^1(\hat{\C}) \ :  \ \int_\C h(x)g(x) \, d^2x=0\} $ 
\begin{equation}\label{formulaT}
T(h)= h+ \ln g+ \sum_{k=1}^n \chi_k G(z_k,.)- \ln \left (  \int_\C e^{h(z)+\sum_{k=1}^n \chi_k G(z_k,z)} g(z) \dd^2z  \right ).
\end{equation}

It is known since the work of Troyanov \cite{troyanov} that under condition \eqref{Troya} there exists a solution to \eqref{Liouv}+\eqref{lcfasympt} provided that $\sum_{k=1}^n \chi_k-4$ and $\Lambda$ have same sign (where in the degenerate case  $\sum_{k=1}^n \chi_k-4=0$ this amounts to $\Lambda=0$). Under Troyanov's condition  \eqref{Troya} and using integration by parts, any solution $\phi$ satisfies
\begin{equation*}
\sum_{k=1}^n\chi_k-4 = \Lambda   \left (  \int_\C  e^{\phi(z)} \dd^2z \right )
\end{equation*}
and hence $\phi= h+ \ln g+ \sum_{k=1}^n \chi_k G(z_k,.)$ where 
$h$ satisfies the following equation
\begin{equation}\label{newequation}
\Delta_{g} h = 2 \pi \left( \sum_{k=1}^n\chi_k-4   \right) \left ( \frac{e^{h+  \sum_{k=1}^n \chi_k G(z_k,.)}}{  \int_\C  e^{h(z) + \sum_{k=1}^n \chi_k G(z_k,z) } g(z) \dd^2z} -\frac{1}{4 \pi }   \right ) .
\end{equation} 
where $\Delta_{g} h (x)= \frac{1}{g(x)} \Delta_x h(x)$ is the Laplacian in the round metric $g$. Now, considering the decomposition  \eqref{unitvol} of the shifted version of $\phi$ corresponding to unit volume, we obtain that  $\phi-\ln (\int_{\C} e^{\phi(z)} \dd^2z)=T(h)$ where $h$ is  the 
 solution  to \eqref{newequation} in $\bar H^1(\bbC)$.

Note that such solutions to \eqref{newequation} in $\bar H^1(\bbC)$ can be obtained as critical points of the following action 
\begin{equation}\label{criticalaction}
J(h)= \frac{1}{4 \pi}  \int_{\C} |\nabla_{z} h(z)|^2\, \dd^2z  + (\sum_{k=1}^n \chi_k-4  ) \ln \int_{\C}  e^{h(z)+\sum_{k} \chi_k G(z_k,.)} \, g(z) \dd^2z.
\end{equation} 
Therefore, on the quantum level, it is natural to consider the image under transformation $T$ \eqref{formulaT} of the measure formally defined on $\bar H^1(\bbC)$ by
\begin{equation*}
e^{-J(h)}   Dh.
\end{equation*}
This is precisely what is achieved by conditioning the Liouville field defined by  \eqref{defLiouvillefield} to have volume $1$ and which leads to formula \eqref{UnitvolumeLiouville}.
For \eqref{UnitvolumeLiouville} to be well defined   we only need to require $\E\big[   Z_0^{ -\frac{\sum_k \alpha_k-2Q}{\gamma} }  \big] < \infty$ and this is equivalent to the following bounds which are the quantum analogues of Troyanov's condition \eqref{Troya}
\begin{equation}\label{extendedSeiberg}
\forall k, \; \alpha_k <Q, \quad \quad    2Q-\sum_{k=1}^n \alpha_k < \frac{4}{\gamma}  \wedge \min_{k=1}^n  2 (Q-\alpha_k).
\end{equation}

\noindent In conclusion, when our parameters $\chi_k$ satisfy
$$\sum_{k=1}^n \chi_k  \le 4$$ 
and  provided that \eqref{Troya} is satisfied, the unit volume framework enables to investigate the  semiclassical asymptotic with $\alpha_k= \frac{\chi_k}{\gamma}$ with $\chi_k$ fixed  and $\gamma$ going to $0$. 
 Unfortunately, there are some technical obstructions for our proof to cover also the positive curvature case $\sum_{k=1}^n \chi_k  < 4$. Indeed
we would need to extend Lemma \ref{expintegra} below with $\alpha<0$ and this is currently out of reach with our method. Moreover, it could be the case that \eqref{criticalaction} admits several critical points in $\bar H^1(\bbC)$ and therefore equation \eqref{newequation} has several solutions; nonetheless, the work of Liu-Tian \cite{LG} ensures unicity in the case where $\chi_k >0$ for all $k$ (notice that this condition is automatically satisfied for $n=3$) and therefore  \eqref{criticalaction} admits a unique minimum in that case. It seems reasonable to conjecture that the unit volume Liouville field will converge to the unique minimum of \eqref{criticalaction} when it exists\footnote{To the best of our knowledge, this point is not known and could be false in full generality assuming only \eqref{Troya}.} and that a large deviation principle will hold with the (non convex) functional  \eqref{criticalaction}.  Finally let us mention the works of Eremenko \cite{Eremenko} and Mondello-Panov \cite{MP,MP1} where the authors construct solutions to \eqref{newequation} when condition \eqref{Troya} is not necessarily satisfied. However, there exists presently no quantum analogue of these constructions.

\subsubsection*{The semiclassical limit with two conical singularities}

Another case where classical and quantum Liouville theory can be constructed without the negative curvature assumption is the case of metrics with two conical singularities of same weight.  
Indeed, one can also construct metrics with positive curvature and two conical singularities at $0$ and $\infty$ with weight $\chi \in [0,2)$. If the metric has unit volume than all the solutions are given by
\begin{equation}\label{unitvolumemetric}
\frac{2-\chi}{2\pi}  \lambda^2 \frac{|\lambda z|^{- \chi}}{(1+|\lambda z|^{2- \chi})^2} |\dd z|^2
\end{equation}
where $\lambda>0$. In a recent paper, Duplantier-Miller-Sheffield \cite{DMS} introduced the quantum analogue of these measures (more specifically, they introduced the quantum analogue of the round metric which corresponds to $\chi=0$ and to $\alpha=\gamma$ in the following discussion). More precisely, they introduced an equivalence classe of random measures (defined up to dilations and rotations) with two marked points  $0$ and $\infty$. The random measures are defined on the cylinder $\R \times [0,2\pi]$ and we will identify the cylinder with the Riemann sphere via the conformal mapping $z \mapsto e^{-z}$. If $\alpha \in (\frac{\gamma}{2},Q)$ then we introduce 
\begin{equation*}
\mathcal{B}^\alpha_s = \left\{
 \begin{array}{ll}
  B^\alpha_{-s} & \text{if } s < 0\\
    \bar{B}^\alpha_{s} & \text{if } s >0 \end{array} \right. \ 
    \end{equation*}
where $B^{\alpha}_s,\bar B^{\alpha}_s$ are  two independent Brownian motions with negative drift $\alpha-Q$ and conditioned to stay negative. Let $Y$ be a log-correlated Gaussian field with covariance 
\begin{equation*}  
\E[  Y(s,\theta) Y(t,\theta') ]  = \ln \frac{e^{-s}\vee e^{-t}}{|e^{-s}e^{i \theta} - e^{-t} e^{i \theta'} |}
\end{equation*}
and associated measure
\begin{equation*}
N_\gamma(\dd s \dd \theta):= e^{\gamma Y(s,\theta)-\frac{ \gamma^2 }{2}E[Y(s,\theta)^2]}  \dd s \dd \theta
\end{equation*}

The (unit area or volume) $\alpha$-quantum sphere is the unit volume random measure $\mu(ds d\theta)$ defined on the cylinder $\R \times [0,2\pi]$  by
\begin{equation*}
\E[  F(\mu(\dd s \dd \theta)) ]=   \frac{\E[    F(  \frac{e^{   \gamma \mathcal{B}_s^\alpha  }  N_\gamma(\dd s \dd \theta)}{\rho(\alpha)} )    \rho(\alpha)^{\frac{2}{\gamma}  (Q-\alpha)}      ]}{  \bar{R}(\alpha) }
\end{equation*}
where 
\begin{equation*}
\rho(\alpha)= \int_{-\infty}^\infty  e^{   \gamma \mathcal{B}_s^\alpha  } N_\gamma(\dd s \times [0,2 \pi]) 
\end{equation*}
and $\bar{R}(\alpha)$ is the Liouville reflection coefficient defined by
\begin{equation*}
\bar{R}(\alpha)= \E[   \rho(\alpha)^{\frac{2}{\gamma}  (Q-\alpha)}  ].
\end{equation*}
For $\alpha= \frac{\chi}{\gamma}$ with $\chi \in (0,2)$ fixed,  we conjecture that   the $\alpha$-quantum sphere (mapped back to the Riemann sphere)  converges (as $\gamma \to 0$)  to the positive curvature metric given by \eqref{unitvolumemetric} for some $\lambda>0$.


\subsection{Organization of the paper}

The rest of the paper is organized as follows. In the next section, we introduce general tools and notations on Gaussian variables; we will also give an alternative formula for the Liouville action $S_{(\chi_k,z_k)}$ which is more adapted to our framework. In section 4, we will state and prove Proposition \ref{prop:asymp} which is the key result of the paper; from Proposition  \ref{prop:asymp} , we will deduce in the remainder of section 4 all the main results of the paper (which were stated in Section 2.2). In Section 5, we prove technical results which are used in the proof of Proposition \ref{prop:asymp}. Finally, in the Appendix, we gather convexity considerations and general large deviation type results.

\section{Technical preliminaries }

Let us introduce in this section a few technical tools which we we use to prove our main results. This includes some basic notions concerning the Wick notation which we have used in a couple of equations in the previous Section, as well some classical results concerning Gaussian processes.

\subsection{Wick Notation}\label{wick}

If $Z$ is a Gaussian variable with mean zero and variance $\sigma^2$, its Wick $n$-th power ($n\in \N$) is defined by 
\begin{equation}
:Z^n: \,\,= \sum_{m=0}^{\lfloor  n/2 \rfloor}\frac{(-1)^m n!}{m!(n-2m)! 2^m} \sigma^{2m} Z^{n-2m}= \sigma^n H_n(\sigma^{-1} Z)
\end{equation}
where $H_n$ is the $n$-th Hermite Polynomial. 
This definition is designed to makes  the Wick monomials orthogonal to one another. More precisely if $(Z,Y)$ is a Gaussian vector we have
\begin{equation}\label{scalar}
\E\left[:Z^n: :Y^m: \right] = n! \ind_{n=m}\E\left[ZY\right]^n.
\end{equation}
The Wick exponential is defined formally as the result of the following expansion in Wick powers
\begin{equation}\label{supertaylor}
:e^{\gamma Z}:\,\, = \sum_{n=0}^{\infty} \frac{\gamma^n :Z^n:}{n!}= \exp\left( \gamma Z-\frac{ \sigma^2 \gamma^2}{2}\right).
\end{equation}
In the present paper, we mostly use the notation $\int_\C :Y(x)^2:  \mu(\dd^2 x)$ and $\int _\C :e^{\gamma Y(x)}: \mu(\dd^2x)$  for Gaussian fields defined on $\bbC$ and $ \mu(\dd^2x)= \rho(x)  g(x)\dd^2 x$ where $\rho$ satisfies for some $\eta>0$,
\begin{equation}\label{normalization}
\int_{\C} \rho(x)^{(1+\eta)}  g(x)\dd^2 x< \infty.
\end{equation} 
While these integrals makes sense when $(Y(x))_{x\in \bbC}$ is a field with uniformly bounded covariance, some additional care is needed when we use the notation for distributional fields such as the GFF.

\subsubsection{Wick Notation for Gaussian Fields}

Let us consider $Y$ a Gaussian field on $\bbC$ (or a subset)
whose covariance satisfies 
\begin{equation}\label{cozz}
 \bbE[ Y(x) Y(y)]:= \log\frac{1}{|x-y|} + \log (1+|x|) + \log(1+|y|) + O(1)
 \end{equation}
Consider   a sequence  $(Y_n)_{n\ge 1}$ of Gaussian fields defined on the same space as $Y$ and such that the full process $[(Y_n),Y]$ is Gaussian. Assume this sequence has bounded covariance and converges to $Y$ in the following sense 
\begin{equation}\begin{split}\label{weaksense}
\forall n\ge 1, \forall x,y\in \bbC, \quad & |\bbE[ Y_n(x) Y_n(y)]|\le C + \bbE[ Y(x) Y(y)],\\
\forall u\in C^{\infty}_c(\bbC) \quad  & \lim_{n\to \infty}\int_\C (Y_n(x)-Y(x)) u(x)\dd^2 x=0,
     \end{split}
\end{equation}
where the first inequality has to be satisfied for an arbitrary constant $C>0$ that does not depend on $n$, and the convergence in the second line is in probability.
An example of sequence satisfying these conditions is a convolution sequence   as the one described in Section \ref{sec:backgr}. 
It can be checked via elementary computations that for any fixed $k$ the sequence  
\begin{equation}
\int_{\bbC} :Y_n^k(x): \mu(\dd^2x),
\end{equation}
is Cauchy in $\bbL_2$. We can thus define 
$\int_{\bbC} :Y^k(x):  \mu(\dd^2x)$ as the limiting random variable which does not depend on the sequence $(Y_n)_{n\ge 1}$.

\medskip

The distribution $:e^{\gamma Y}:$ can be defined using  the procedure described in Equations \eqref{circlegreen}-\eqref{law} (and detailed e.g.\ in \cite{Ber}) as soon as $\gamma<2$).
Let us however provide a concise and self-contained argument which asserts the existence of $\int_{\bbC} :e^{\gamma Y(x)}:  \mu(\dd^2x) ,$
as soon as $\gamma^2(1+\eta^{-1})<2$.
It can be checked that if $Y_n$ satisfies \eqref{weaksense} the sequence $\int_{\bbC} :e^{\gamma Y_n(x)}:  \mu(\dd^2 x)$, is Cauchy in $\bbL_2$ provided that 
$$ \int_{\bbC^2} \left(\frac{(1+|x|)(1+|y|)}{|x-y|}\right)^{\gamma^2} \mu(\dd^2 x) \mu(\dd^2 y)<\infty.$$
H\"older's  inequality and \eqref{normalization} guarantees that the above holds as soon as $\gamma^2(1+\eta^{-1})<2$.

\subsection{Gaussian space tools}

\subsection*{Girsanov/Cameron Martin Formula}

The formula states how the distribution of elements of a Gaussian Hilbert space are modified by an exponential tilt of a random variable. In our context it says that if $Y$ is a centered Gaussian field with covariance funtion $K(\cdot, \cdot)$ with displays a logarithmic divergence (similar to \eqref{cozz}) then we have for any bounded continuous fonction on $H^{-1}(\bbC)$ and any signed measure $\mu$ such that $\int_{\C^2} K(x,y) \mu( \dd^2 x)  \mu(\dd^2 y)<\infty$
\begin{equation}
\bbE[ F(Y) e^{\int_{\C} Y(x)  \mu(\dd^2 x) -\frac{1}{2} \int_{\C^2} K(x,y) \mu(\dd^2 x)  \mu(\dd^2 y)}]= \bbE[ F(Y+ \int_\C K(\cdot, y) \mu(\dd^2 y)].
\end{equation}
The formula is easily checked for finite dimensional marginals and then extended by continuity.
We are going to apply this formula also to $:e^{\gamma Y}:$ which is not a continuous fonction of $Y$. However $:e^{\gamma Y_\gep}:$ is, and  using the limiting procedure \eqref{Meps} we can deduce from the above that 
\begin{align}
& \bbE[ F(Y, \int_\C :e^{\gamma Y(x)}:  \nu(\dd^2x) ) e^{\int_\C Y(x) \mu(\dd^2x) -\frac{1}{2} \int_{\C^2} K(x,y) \mu(\dd^2x)  \mu(\dd^2 y)}]  \nonumber  \\
& = \bbE[ F(Y+ \int_{\C} K(\cdot, y) \mu(\dd^2 y), \int_{\C} e^{ \int_\C K(x, y) \mu( \dd^2 y)}  :e^{\gamma Y(x)}:  \nu(\dd^2 x) )   ] \label{girsanov}
\end{align}
for every $F$  continuous on $H^{-1}(\bbC)\times \bbR$,  $\gamma<2$ and $\nu$ with density w.r.t Lebesgue and finite total mass.
\subsection*{Positive association for positively correlated fields}

A classical result of Gaussian analysis 
\cite{Pitt} states that if $(X_i)_{i\in I}$ ($I$ finite) is a Gaussian vector such that $\bbE[X_i X_j]\ge 0$ for all $i,j\ge 1$ then for any pair of  square integrable functions $f, g: \bbR^I \to \bbR$ which are non-decreasing in all $|I|$ variables we have 
\begin{equation}\label{FKG}
 \bbE\left[ f((X_i)_{i\in I})g((X_i)_{i\in I})\right] \ge \bbE[f((X_i)_{i\in I})]\bbE[g((X_i)_{i\in I})].
\end{equation}
In order to apply this inequality to our field which are indexed by $\bbC$ and defined in a space of distribution we simply apply a limiting procedure.


\subsection{White Noise Decomposition}\label{decompex}

While it is a priori possible to write a proof of our results by working directly on the Riemann Sphere $\hat \bbC$, it turns out to be more convenient for notation to work with a field  defined in the ball $B(0,1):= \{z \ : \ |z|\le 1\}$.
Instead of working directly with the restriction of $X$ on $B(0,1)$ we are going to look at a randomly shifted version of it which we denote by $\tilde X$ that
possesses the convenient feature of having an explicit white noise decomposition for which correlations are positive on all scales, which is helpful in view of using positive association.

\medskip

As an intermediate step we introduce $\bar{X}$ the GFF on the plane with average zero on the circle of radius $1$. 
It can be obtained by considering the limit  $ X- \lim_{\epsilon \to 0}(2\pi)^{-1}\int_{0}^{2\pi} X_{\epsilon}(\theta) \dd \theta$ where $X_{\epsilon}$ is the mollified version of $X$ considered in Section \ref{sec:backgr}.
As $\int_\C X(x)   g(x) \dd^2x=0$ we also have   
\begin{equation}\label{moinlamoy}
X=\bar{X}- \frac{1}{4\pi }\int_{\C}  \bar{X}(x)  g(x) \dd^2x.
\end{equation}
The covariance of $\bar{X}$ in the ball $B(0,1)$ is given by
$\E[\bar{X}(x)\bar{X}(y)]= \ln \frac{1}{|y-x|}$
as can be checked by a simple computation of covariances.
Now to obtain a positively correlated field, we  set  $\tilde{X}:=\bar{X}+\sqrt{\ln 2}\,Y$   where $Y$ is an independent standard Gaussian variable. Equation \eqref{moinlamoy} is also satisfied with $\bar X$ replaced by $\tilde X$ and the covariance of this last field satisfies \cite[Example 2.3]{cf:RoVa}
\begin{equation*}
\E[\tilde{X}(x)\tilde{X}(y)]= \ln \frac{2}{|y-x|}= 2\int^{\sqrt{2}}_0 (t-\sqrt{|x-y|})_+ \frac{ \dd t }{t^2}+ \sqrt{2}(\sqrt{2}-\sqrt{|x-y|})_+.
\end{equation*}  
Instead of using a convolution to approximate $\tilde X$ by a smoothened field, we construct it as a limit of functional fields.
We let $(\tilde X_t(x))_{x\in B(0,1),t\ge 0}$ be a bivariate field of covariance 
\begin{equation}
 \bbE[ \tilde X_s(x) \tilde X_t(y) ]= 2 Q_0(x,y)+\int^{t\wedge s}_0 Q_u(x,y),
\end{equation}
where 
\begin{equation} 
 Q_u(x,y):=  (1- \sqrt{e^u |x-y|/2})_+.
  \end{equation}
Note that $Q_u(x,y)$ is a positive definite function \cite{cf:PaYu}.
We have in particular that, letting $K_t(x,y):= \int^{t}_0 Q_u(x,y)+2Q_0(x,y)$ denote the covariance function of the field $\tilde X_t$, there exists a constant $C$ such that for all $x,y\in B(0,1)$, $t\ge 0$ we have
\begin{equation*}
\forall x, y \in B(0,1), \quad   |K_t(x,y)- \max (\log|x-y|, t)|\le C.
\end{equation*}  

\subsection{The centered Liouville action $S_{L, (\chi_k,z_k)}$} 
 In what follows, it will be convenient to introduce the notation
\begin{equation}\label{defw}
w(z)=  e^{\sum_{k=1}^n \chi_k G(z_k,z)}.
\end{equation}

Note that in our setup, the field $\gamma \phi$ displays logarithmic singularities at $(z_k)$, cf.\ \eqref{defLiouvillefield}, and these 
singularities persists in the semiclassical limit.
As it is easier to work with a field with no such singularities, we 
replace $\phi$ by a more regular function $h$ as done in Section \ref{ops} above. 
We introduce thus the centered Liouville action $S_{L, (\chi_k,z_k)}$ on $H^1(\bbC)$ which roughly  corresponds to rewriting $S_{(\chi_k,z_k)}(\phi)$ as a function of $h=\phi- \ln (gw)$.
 It admits the following expression
 \begin{equation}\label{defcentered}
S_{L,(\chi_k,z_k)}(h)=  \frac{1}{4\pi}\int_\C ( |\nabla_z h(z)|^2 + 4 \pi \Lambda e^{h(z)}w(z)  g)\dd^2z+   \frac{1}{4 \pi  }(4- \sum_{k=1}^n\chi_k) \int_{\C} h(z) g(z) \dd^2z.
\end{equation}
One can extend $S_{L,(\chi_k,z_k)}$ to $H^{-1}(\hat \C)$ by setting it to be equal to infinity outside $H^1(\hat \C)$: the extension is convex and a good rate function in the terminology of large deviation theory.
The following claim (proved in Appendix \ref{app:equiv}) motivates our definition:

\begin{lemma}\label{lemmaequivalence}
Given  $h \in H^{1}(\hat \bbC)$ and setting
\begin{equation*}
 \phi := h+ \ln (gw),
\end{equation*}
the following identity holds
\begin{equation}  \label{mainidentityapp}
 S_{(\chi_k,z_k)}(\phi)     = S_{L,(\chi_k,z_k)}(h)+  l(  (\chi_k,z_k) )  +C_\star((\chi_k)) 
\end{equation}
where $C_{\star}((\chi_k))$ is defined by  \eqref{Theconstant},
 and
 \begin{equation}\label{Thelfunction}
l(  (\chi_k,z_k) )=  -\sum_{k=1}^n \chi_k(1-\frac{\chi_k}{4}) \ln   g(z_k)  - \frac{1}{2} \sum_{k \not = j }  \chi_k \chi_j G(z_j, z_k) 
\end{equation}

\end{lemma}

\section{Reducing the problem to partition function asymptotics}

\subsection{Introducing the statement}
The goal of this Section is to reduce the proof of our three main results  Proposition \ref{asymppart}, Theorem \ref{thSemiclassical}
and Proposition \ref{propLargedev} to a general statement.
Let us recall that  we always assume that 
$\alpha_k=\chi_k \gamma^{-1}$, $\mu=\gL \gamma^{-2}$.
\medskip

A statement concerning large deviations can be obtained by studying the asymptotic behavior of the Laplace transform of the field which is given by
\begin{equation}
 \lim_{\gamma \to 0} \frac{1}{\gamma^2} \log \frac{\langle e^{\gamma^{-1} \int_\C \psi(x) \phi(x) g(x) \dd^2x   }   \prod_{k=1}^n V_{\alpha_k}(z_k)   \rangle_{\gamma,\mu}}{ \langle  \prod_{k=1}^n V_{\alpha_k}(z_k)   \rangle_{\gamma,\mu}}.
\end{equation}
for appropriate $\psi$. On the other hand, to obtain results concerning the limiting law of $\phi- \gamma^{-1} \phi_*$, we need to compute the following limit for all bounded continuous function on $H^{-1}(\bbC)$
\begin{equation}
 \lim_{\gamma \to 0} \frac{ \langle F(\phi- \gamma^{-1} \phi_*)   \prod_{k=1}^n V_{\alpha_k}(z_k)   \rangle_{\gamma,\mu}}{ \langle  \prod_{k=1}^n V_{\alpha_k}(z_k)   \rangle_{\gamma,\mu}}.
\end{equation}
Hence we can prove both statements if we obtain 
sharp asymptotics for 
$$\langle  e^{\gamma^{-1} \int_\C \psi(x) \phi(x)  g(x) \dd^2x} F(\phi- \gamma^{-1} \phi_{*,\psi})   \prod_{k=1}^n V_{\alpha_k}(z_k)   \rangle_{\gamma,\mu},$$
where $\phi_{*,\psi}$ is a function to be determined but which coincides with $\phi_*$ when $\psi\equiv 0$.
Using our factorisation of the measure into $X$ and $a$ given by \eqref{defLiouvillefield} we can in fact compute a separate asymptotic for $a$ and $X$. 

\medskip

Before discussing things in more details let us introduce further notations used in this section. 
Considering $\psi$ a smooth function on $\hat \bbC$, we set
\begin{equation}\label{def:cphi}
c_\psi= \int_{\C}  \psi(x) g(x) \dd^2x+\sum_{k=1}^n \chi_k-4 . 
\end{equation} 
We also introduce $h_\psi\in   \bar H^1(\hat{\C}):=\{ h\in H^1(\hat{\C}) \ :  \ \int_\C h(x) g(x) \dd^2x=0\} $  the unique solution to the Liouville equation     (see appendix \ref{app:exist})
\begin{equation}
\begin{cases}
 \label{equationdebase1}
\Delta_{g} h &= -2 \pi \left(\psi- \frac{1}{4\pi} \int_\C \psi(x)  g(x) \dd^2x \right)+2 \pi c_\psi \left(\frac{ w e^{h}}{\int _\C w(x) e^{h(x)}   g(x) \dd^2 x  }-\frac{1}{4 \pi}\right),\\
  \int_\C h(x)   g(x) \dd^2x&=0.
\end{cases}
\end{equation}
and set 
\begin{equation}\label{mupsi}
\mu_{\psi}(\dd^2 x):= \frac{ e^{h_{\psi}(x)}w(x)   g(x) \dd^2 x}{\int_\C  e^{h_{\psi}(y)}w(y)   g(y) \dd^2 y}= e^{\bar h_{\psi}} \dd^2 x.
\end{equation} 
Note that $\mu_{\psi}(\dd^2 x)$ has volume $1$.

\medskip

The asymptotic for the integral in $a$ is a standard computation.
Using the following variant of Stirling's formula  
$$\gG(x-1)=\gG(x+1)/(x(x-1)) \sim \sqrt{2\pi} x^{-3/2} \left(\frac{x}{e}\right)^x,$$ 
we have 
\begin{multline}
\int_{\bbR} e^{a (s_{\gamma}+\frac{\int_\C \psi(x)   g(x) \dd^2 x}{\gamma^2})} e^{-\mu e^a} \dd a=
\int_{\bbR} e^{a\left(\frac{c_{\psi}}{\gamma^2} -1\right)} e^{- \frac{\gL}{\gamma^2} e^a} \dd a\\
=\left( \frac{\gamma^2}{\gL} \right)^{\frac{c_{\psi}}{\gamma^2} -1}\gG\left(\frac{c_{\psi}}{\gamma^2} -1\right)
= \gamma\gL\sqrt{\frac{2\pi}{c_{\psi}^3}} e^{\frac{c_{\psi}}{\gamma^2}[\ln (c_{\psi}/\gL)-1]}(1+o(1)).
\end{multline}
Moreover (and this is only of interest when $\psi=0$),
we have for any bounded continuous function $F$ on $\bbR$
\begin{multline}\label{ziax}
 \int_{\bbR} F(\gamma^{-1}a-\gamma^{-1}\log(c_{\psi}/\gL))e^{a( s_{\gamma}+\frac{\int_\C \psi(x) g(x) \dd^2x}{\gamma^2})} e^{-\mu e^a} \dd a \\=
\gamma\gL\frac{\sqrt{2\pi}}{c_{\psi}} e^{\frac{c_{\psi}}{\gamma^2}[\ln (c_{\psi}/\gL)-1]} \left(\int_{\bbR} F(x) e^{-\frac{c_\psi x^2}{2}} \dd x +o(1)\right)
\end{multline}
showing that after recentering the variable $\gamma^{-1} a$ converges to a Gaussian
of variance $c_{\psi}^{-1}$. 

\medskip

The part concerning $X$ is the main probabilistic estimate of the paper. Given $F$ a continuous bounded function in $H^{-1}(\bbC)$ we want to determine the precise asymptotic of the following Laplace functional
\begin{multline*}
L_\gamma (\psi,F) =\E\left[ F(X- \gamma^{-1}[ \log  Z_0 + \bar h_\psi ])  e^{\frac{1}{\gamma}  \int_\C  \psi(x)  (X(x)-\gamma^{-1} \ln Z_0) g(x) \dd^2x  } Z_0^{-s}  \right]  \\
=\E\left[ F(X- \gamma^{-1}[ \log  Z_0 +\bar h_\psi]) e^{\frac{1}{\gamma}  \int_\C  \psi(x)  X(x) g(x) \dd^2x  }  (Z_0)^{1-\frac{c_{\psi}}{\gamma^2}} \right].
\end{multline*}

\begin{proposition}\label{prop:asymp}
If $\psi$ is such that $c_{\psi}>0$, then we have in the small $\gamma$ asymptotics  
\begin{align*}
 & L_\gamma (\psi,F)  \stackrel{\gamma \to 0}{=} e^{\gamma^{-2} H(\psi)}(\int_\C e^{h_{\psi}(x)} w(x)  g(x) \dd^2x) \\ 
 & \times \left(\E[  F(X-\int_\C X(x) \mu_{\psi}(\dd^2x))  e^{\frac{c_{\psi}}{2} \left( (\int_\C X(x) \mu_{\psi} (\dd^2x))^2- \int_\C :X^2(x):  \mu_{\psi}(\dd^2x) \right)} ]+o(1)\right)
\end{align*}
with $$H(\psi):=-\frac{1}{4 \pi } \int_\C  |\nabla_{x} h_\psi(x)|^2 \dd^2x -c_\psi \ln \int_\C w(x) e^{h_\psi(x)} g(x) \dd^2x+ \int_\C h_\psi(x) \psi(x) g(x) \dd^2x.  $$


\end{proposition}

Now we can combine Proposition \ref{prop:asymp} with Equation \eqref{ziax} to prove our main results.

\subsection{Proof of Proposition \ref{asymppart} and Theorem \ref{thSemiclassical}}

{\bf In what follows and in order to have more concise formulae, we will sometimes write $\dd   g$ in place of $g(x) \dd^2x$ or $\dd \mu^*$ in place of $\mu^*(\dd^2x)$, etc... Also we will simply write $\int$ without indicating the set on which we integrate: this should be clear from the context.}

\vspace{0.1 cm}

To obtain the limit of correlations, we use Proposition \ref{prop:asymp}  and Equation \eqref{ziax} for $\psi \equiv 0$ and $F\equiv 1$. Notice that $\mu_{0}:=\mu_{\psi=0}$  defined by \eqref{mupsi} coincides with  $\mu^*$ defined by \eqref{mustar}  since  
by definition $h_{0}+ \log w   g$ is the solution of \eqref{Liouv}+\eqref{lcfasympt} up to constant, and $c_0:= \sum_{k=1}^n \chi_k -4$.
We obtain, recalling \eqref{funamental} and   \eqref{Thelfunction}
\begin{equation}
\begin{split}
 K({\bf z}) &\stackrel{\gamma\to 0}{\sim}  
 4\gamma^{-1} e^{-\frac{\kappa}{2}(\sum_{k=1}^n \chi_k-4)} e^{\frac{1}{\gamma^2} \left(\frac{\kappa}{2}\sum_{k=1}^n \chi^2_k- 8\kappa-l((\chi_k,z_k))\right)}\\
 \int_{\bbR} e^{a s_{\gamma}} e^{-\mu e^a} \dd a &\stackrel{\gamma\to 0}{\sim}  \frac{\gL\gamma\sqrt{2\pi}}{c_0^{3/2}}e^{\frac{1}{\gamma^2}c_0[\ln (c_0/\gL)-1]},\\
 \bbE[Z_0^{-s}]& \stackrel{\gamma \to 0}{\sim} (\smallint e^{h_0}w \dd  g  )e^{-\frac{1}{\gamma^2}\left( \frac{1}{4 \pi } \int  |\nabla_{z} h_0|^2 \dd^2z+c_0 \ln(\smallint e^{h_0}w \dd  g  ) \right)} \E[   e^{\frac{c_{0}}{2} \left( (\int X \dd \mu^*)^2- \int :X^2: \dd \mu^* \right)} ] 
 \end{split}
\end{equation}
with $h_0$ given by  \eqref{equationdebase1} with $\psi=0$. Altogether we obtain
\begin{multline}
 \langle  \prod_{k=1}^n V_{\alpha_k}(z_k)   \rangle_{\gamma,\mu} 
 \stackrel{\gamma\to 0}{\sim} 
  \frac{4\gL \sqrt{2\pi}}{c_0^{3/2}}e^{\frac{1}{\gamma^2}c_0[\ln (c_0/\gL)-1]}   e^{-2\kappa-\frac{\kappa}{2}\sum_{k=1}^n \chi_k} e^{\frac{1}{\gamma^2} \left(\frac{\kappa}{2}\sum_{k=1}^n \chi^2_k- 8\kappa-l((\chi_k,z_k))\right)}\\
  \times (\smallint e^{h_0}w \dd  g  )e^{-\frac{1}{\gamma^2}\left( \frac{1}{4 \pi } \int  |\nabla_{z} h_0|^2 \dd^2z+c_0 \ln(\smallint e^{h_0}w \dd  g  ) \right)} \E[   e^{\frac{c_{0}}{2} \left( (\int X \dd \mu^*)^2- \int :X^2: \dd \mu^* \right)} ] .\label{rem:equiv}
\end{multline}
 In the appendix \ref{app:exist}, we prove that the quantity (see \eqref{defJ})
$$J_0(h_0):= \frac{1}{4 \pi } \int  |\nabla_{z} h_0|^2 \dd^2z+c_0 \ln(\smallint e^{h_0}w \dd  g) $$
can be related to the quantity $\min_hS_{L,(\chi_k,z_k)}(h)$ by the relation
$$\min_hS_{L,(\chi_k,z_k)}(h)=-c_0[\ln (c_0/\gL)-1]+J_0(h_0).$$ 
This is the content of Proposition \ref{prop:legendre}. Consequently Lemma \eqref{lemmaequivalence} allows us express $J_0(h_0)$ as
$$J_0(h_0)= \min_\phi S_{(\chi_k,z_k)}(\phi)-l((\chi_k,z_k))-C_\star((\chi_k)) +c_0[\ln (c_0/\gL)-1].$$
Furthermore, integrating \eqref{unitvol} on the sphere and using
$$ \int \log w \dd  g=0 \quad \text{ and } \quad \frac{1}{4\pi} \int \log  g \dd   g= 2(\ln 2 -1),$$
we get that 
$$ \int e^{h_0}w \dd   g=\int e^{\phi_*-\frac{1}{4\pi}\phi_*}\dd^2 z \,e^{2(\ln 2 -1)}.$$
By plugging these relations into \eqref{rem:equiv}, we get the statement of Proposition  \ref{asymppart}.


\bigskip
For Theorem \ref{thSemiclassical}, we can perfom the same computation including a function.
Recalling \eqref{defLiouvillefield} we obtain from Proposition \ref{prop:asymp} that under $\bbP_{\mu,(\alpha_k,z_k)}$ in the semi-classical limit $\gamma X-\log Z_0$ converges to $\bar h_0$, and from \eqref{ziax} that $a$ converges to $\log (c_{0}/\gL)$.
Hence 
$$\gamma \phi= \gamma X -\log Z_0+a +\log (w  g) + \gamma^2 \left( \frac{\log   g}{4} + \frac{\kappa}{2}\right)$$
converges to $ \bar h_0 + \log (c_{0}/\gL)+ \log (w g)$.
Note that up to a constant shift, the above function is equal to $\phi_*$.
To check that the involved constant is $0$, it is sufficient to observe that 
\begin{equation}
 \int e^{\bar h_0 + \log (c_{0}/\gL)+ \log (w  g)} \dd z= \frac{c_0}{\gL}= \int e^{\phi_*} \dd z.
\end{equation}
Concerning the convergence of $\phi-\gamma^{-1} \phi_*$,
the corresponding limit corresponds to the independent sum of  $ \gamma^{-1}(a- \log (c_{0}/\gL))+(X-\gamma^{-1}(\log Z_0+ \bar h_0))$. Equation \eqref{ziax} imples that the first term converges to a Gaussian of variance $c^{-1}_0$ while Proposition \eqref{prop:asymp}
guarantees the convergence of the second term to the prescribed field.

\subsection{Proof of Proposition \ref{propLargedev}}

According to relation \eqref{mainidentityapp}, the claim is equivalent to proving that the field $\gamma X=\gamma \phi-\sum_k \chi_k G(z_k, \cdot)- \ln g- \ln \frac{\sum_k \chi_k-4}{\Lambda} $ satisfies a large deviation principle on $H^{-1}(\hat \C)$ with good rate function the centered Liouville action (shifted by its minimum) $S_{L,(\chi_k,z_k)}(.)-S_{L,(\chi_k,z_k)}(h_{\ast}) $. Here we adopt the notations of section \ref{Convexconsi} in the appendix. With these notations, Proposition \ref{prop:asymp} implies straightforwardly for all $\psi$ that 
\begin{equation*}
\frac{1}{\gamma^2}  \ln  \E_{\frac{\Lambda}{\gamma^2},(\frac{\chi_k}{\gamma}, z_k)}[e^{\int_{\C} \psi (z) X(z) g(z) \dd^2 z    }]    \underset{\gamma \to 0}{\rightarrow}    \textsc{f} (\psi)- \textsc{f} (0). 
\end{equation*}
In the language of large deviation theory, $\textsc{f} (\psi)- \textsc{f} (0)$ is the limit of the logarithmic moment generating function of $\gamma X$ (under $\P_{\frac{\Lambda}{\gamma^2},(\frac{\chi_k}{\gamma}, z_k)}$). In Proposition \ref{expintegra}, we prove that the  distribution of $\gamma X$ is exponentially tight under $\P_{\frac{\Lambda}{\gamma^2},(\frac{\chi_k}{\gamma}, z_k)}$. The Legendre transform of $\textsc{f} (\psi)- \textsc{f} (0)$ satisfies $(\textsc{f} -\textsc{f} (0)) ^{\ast}(h)= S_{L,(\chi_k,z_k)}(h)-S_{L,(\chi_k,z_k)}(h_{\ast}) $. Thanks to Lemma \ref{lemmaexposed} on the exposed points of $S_{L,(\chi_k,z_k)}(.)$, we can conclude by using Baldi's theorem in \cite{dembo}: see theorem 4.5.20 page 157.

\subsection{Proof of relation \eqref{Polyacc}}
 
Here we prove relation \eqref{Polyacc} by using the BPZ differential equations established in \cite{KRV}.  
We have the following BPZ differential equation
\begin{align*}
 & \frac{4}{\gamma^2}\partial_{zz}^2\langle   V_{- \frac{\gamma}{2}}(z) \prod_l V_{\alpha_l}(z_l)   \rangle_{\gamma,\frac{\Lambda}{\gamma^2}}   + \sum_k \frac{\Delta_{\alpha_k}}{(z-z_k)^2}  \langle  V_{- \frac{\gamma}{2}}(z) \prod_l  V_{\alpha_l}(z_l)  \rangle_{\gamma,\frac{\Lambda}{\gamma^2}}  \\  
 & + \sum_k \frac{1}{z-z_k}  \partial_{z_k} \langle  V_{- \frac{\gamma}{2}}(z) \prod_l  V_{\alpha_l}(z_l)  \rangle_{\gamma,\frac{\Lambda}{\gamma^2}}     =  0  ,
\end{align*}
where $V_\alpha(z)= e^{\alpha (X(z)+\frac{Q}{2} \ln g(z)+c)}$ and $\Delta_\alpha= \frac{\alpha}{2}  (Q-\frac{\alpha}{2})$. We write $\alpha_k= \frac{\chi_k}{\gamma}$ and set $\eta>0$ small. We consider smooth functions $u_k$ with compact support in $B(z_k,\eta)$ and $u$ a smooth function with compact support in $\C \setminus \cup B(z_k,\eta)$. Using \eqref{funamental} and \eqref{defLiouvillefield}, we have for any smooth function $f$ with compact support in $\C \setminus \cup B(z_k,\eta)$ and $z'_k \in B(z_k,\eta)$ the following identity (the definition of $K$ is given by  \eqref{Kdef}) 
\begin{align*}
&  \langle   \left ( \int_{\C} f(z)  V_{- \frac{\gamma}{2}}(z) \dd^2 z   \right ) \prod_{l=1}^n  V_{\alpha_l}(z'_l)  \rangle_{\gamma,\frac{\Lambda}{\gamma^2}} \\
&  = K(\textbf{z}) \:   \int_0^\infty  y^{ \frac{\sum_k \alpha_k-2Q }{\gamma}-1-\frac{1}{2}}  e^{-\frac{\Lambda}{\gamma^2} y}  \dd y  \: \E \left [       \left (   \int_\C f(z)  e^{-\frac{\gamma}{2}  (X(z)  +\frac{Q}{2} \ln g(z)  + \sum_k \alpha_k G(z'_k,z)       }    \dd^2 z   \right)  Z_0^{\frac{1}{2}} Z_0^{-\frac{\sum_k \alpha_k-2Q}{\gamma}}   \right  ] 
\end{align*}
where $e^{-\frac{\gamma}{2} X (z)}$  denotes the limit of $\epsilon^{\frac{\gamma^2}{8}}e^{-\frac{\gamma}{2} X_{\epsilon} (z)}$ as $\epsilon$ goes to $0$. 
By using Proposition \ref{prop:asymp} (in fact a slight extension of the proposition with $F$ given by an appropriate integral of the exponential function), we get the following equivalent (up to a constant)
\begin{align}
&  \E \left [       \left (   \int_\C f(z)  e^{-\frac{\gamma}{2}  (X (z)  +\frac{Q}{2} \ln g(z)  + \sum_k \alpha_k G(z'_k,z)         }   \dd^2 z    \right  )   Z_0^{-\frac{\sum_k \alpha_k-2Q}{\gamma}}   \right  ]  \nonumber \\
&  \underset{\gamma \to 0}{\sim}    \left (  \int_{\C} f(z) e^{-\frac{\phi_* (z)}{2}}  \dd^2z   \right )     e^{-\frac{S_{(\chi_k,z'_k)}(\phi_*  )}{\gamma^2}} \label{firstequiv}
\end{align}
Applying \eqref{firstequiv} with $f(z)= \partial_{zz}^2  u(z) $ and using the fact that \eqref{firstequiv} is uniform for $z'_k \in B(z_k,\eta)$ (and $f$ has compact support in $\C \setminus \cup B(z_k,\eta)$),  we get  by integration by parts
\begin{align*}
& \int_{\C^{n+1}} u(z) \prod_{k=1}^n u_k (z'_k) \partial_{zz}^2\langle   V_{- \frac{\gamma}{2}}(z) \prod_l V_{\alpha_l}(z'_l)   \rangle_{\gamma,\frac{\Lambda}{\gamma^2}}  \dd^2z \prod_{k=1}^n \dd^2 z'_k \\ 
& \underset{\gamma \to 0}{\sim}    \left (  \int_{\C^{n+1}}  \partial_{zz}^2  u(z)  \prod_{k=1}^n u_k (z'_k)   e^{-\frac{\phi_* (z)}{2}} e^{-\frac{S_{(\chi_k,z'_k)}(\phi_*  )}{\gamma^2}}   \dd^2z  \prod_{k=1}^n \dd^2z'_k   \right )      \\
\end{align*}
We also get for all $j$ that (up to the same constant as in the former equivalents)
\begin{align*}
& \int_{\C^{n+1}} u(z) \prod_{k=1}^n u_k (z'_k) \frac{1}{z-z'_j} \partial_{z'_j} \langle   V_{- \frac{\gamma}{2}}(z) \prod_l V_{\alpha_l}(z'_l)   \rangle_{\gamma,\frac{\Lambda}{\gamma^2}}  \dd^2z \prod_{k=1}^n \dd^2 z'_k \\ 
& \underset{\gamma \to 0}{\sim}    - \left (  \int_{\C^{n+1}}  u(z)  \prod_{k \not =j}^{n-1} u_k (z'_k)   \partial_{z'_j}  \frac{ u_j(z'_j) }{ z-z'_j } e^{-\frac{\phi_* (z)}{2}}  e^{-\frac{S_{(\chi_k,z'_k)}(\phi_*  )}{\gamma^2}}   \dd^2z \prod_{k=1}^n \dd^2 z'_k  \right )     \\
& \underset{\gamma \to 0}{\sim}    \frac{1}{\gamma^2} \left (  \int_{\C^{n+1}}  u(z)  \prod_{k \not =j}^{n-1} u_k (z'_k)     \frac{ u_j(z'_j) }{ z-z'_j }   \partial_{z'_j}  S_{(\chi_k,z'_k)}(\phi_*  ) e^{-\frac{\phi_* (z)}{2}}  e^{-\frac{S_{(\chi_k,z'_k)}(\phi_*  )}{\gamma^2}}   \dd^2z \prod_{k=1}^n \dd^2 z'_k  \right )     \\
\end{align*}

One can then conclude by taking the equivalent $\gamma \to 0$ of the BPZ equation integrated with respect to $ u(z) \prod_{k=1}^n u_k (z'_k)  \dd^2z \prod_{k=1}^n \dd^2 z'_k$ and then taking the limit $\eta$ to $0$.

\subsection{Proof of Proposition \ref{prop:asymp}}

For simplicity we write the proof in the case $F\equiv 1$.
We explain how to adapt the proof for general $F$ in the end.

Note that using our notation we have 
\begin{equation}
Z_0  = \int :e^{\gamma X}: w\dd g=  \left(\smallint e^{h_{\psi}} w \dd   g\right)\left(\smallint :e^{\gamma X}:e^{-h_{\psi}} \dd \mu_{\psi}\right)
\end{equation}
and hence we can rewrite the quantity we wish to estimate in the following manner
\begin{equation}\label{woopz}
\E\left[   e^{\frac{1}{\gamma}  \int  \psi  X \dd  g  } Z_0^{1-\frac{c_{\psi}}{\gamma^2}}  \right]
= \left(\smallint e^{h_{\psi}} w \dd   g\right)^{-\frac{c_{\psi}}{\gamma^2}}\E\left[ Z_0 e^{ -\frac{c_{\psi}}{\gamma^{2}} \left( \ln \int :e^{\gamma X}: e^{-h_{\psi}}  \dd \mu_{\psi}- \gamma \int X\dd \mu_{\psi}\right)} e^{\frac{1}{\gamma}\left(\int  \psi  X \dd   g- c_{\psi}\int  X \dd \mu_{\psi} \right) } \right].
\end{equation}
The first part of our proof consists in checking that  the exponential tilt produced by the second exponential factor exactly cancels the $e^{-h_{\psi}}$ present in the first exponent.

\medskip

\medskip
Then we need to check that after taking into account this exponential tilt, the integral converges. This can be achieved by showing convergence in probability of the integrand and uniform integrability.
This is the content of the following Proposition, whose proof is detailed in the next section.

\begin{proposition}\label{expintegra}
Assuming that $\mu$ is a probability on $\bbC$ satisfying \eqref{normalization} and $\inf_{\bbC} \rho >0$, we have
\begin{equation}\label{labound}
 \sup_{\gamma \in (0,1]}\E\left[  e^{-\frac{\alpha}{\gamma^2}  (  \ln \int :e^{\gamma X}: \dd \mu  - \gamma \int X \dd \mu     )    }   \right] 
 =\E\left[  e^{- \alpha  [ \int  :X^2:\dd \mu  -  (\int X \dd \mu )^2     ]} \right]
\end{equation}
Furthermore we have the following almost sure convergence 
\begin{equation}\label{enproba}
\lim_{\gamma \to 0} \gamma^{-2} (  \ln \smallint :e^{\gamma X}: \dd \mu  - \gamma \smallint X \dd \mu)=   \smallint  :X^2:\dd \mu  -  (\smallint X \dd \mu )^2 
\end{equation}
\end{proposition}

\begin{remark}
The upperlimit $\gamma \le 1$ is arbitrary and is set for commodity, the important part of the result being about the behavior near $\gamma$ near $0$.
 The uniform positivity assumption for $\rho$ is present only to simplify the proof of Lemma \ref{lemmanegtaive}.
 We do not believe it to be necessary for the result to hold. With some straightforward scaling argument, it could be replaced by $\inf_{x\in V}\rho(x)>0,$ for some open subset $V\subset \bbC$. 
\end{remark}

Let $Y_{\psi}$ denote the Gaussian variable present in the second exponential in \eqref{woopz} 
$$Y_{\psi}:=\int  \left(\psi- c_{\psi} e^{\bar h_{\psi}} w  \right) X \dd  g .$$
We have 
\begin{equation}\begin{split}
 \var (Y_{\psi})&:= \int  \left(\psi(x)-c_{\psi} e^{\bar h_{\psi}(x)} w(x)\right)
 \left(\psi(y)- c_{\psi} e^{\bar h_{\psi}(y)} w(y)\right) G(x,y)   g(x)   g(y) \dd^2 x \dd^2 y,\\
  \bE [Y_{\psi}X(x)]&:= \int  \left(\psi(y)- c_{\psi}  e^{\bar h_{\psi}(y)} w(y)\right) G(x,y)   g(y)  \dd^2 y.
 \end{split}
\end{equation}
Using the integral version of  \eqref{equationdebase1} we have 
\begin{equation*}
h_\psi(x)=  \int  \left(\psi(y)- c_{\psi}  e^{\bar h_{\psi}(y)} w(y)\right) G(x,y)   g(y)  \dd^2 y.
\end{equation*}
Hence we have $\bE [Y_{\psi}X(x)]=h_{\psi}(x),$ 
and using integration by part and $\int  h_{\psi} \dd    g=0$
\begin{equation}
\var (Y_{\psi})=-\frac{1}{2\pi}\int  h_{\psi}\left[ \Delta_{  g} (h_{\psi}) +\frac{1}{4\pi}(\smallint \psi \dd   g-1)\right]
 \dd    g
  =\frac{1}{2\pi}\int  |\nabla_{z} h_{\psi}|^2  \dd^2z.
\end{equation}
We can thus rewrite \eqref{woopz} in the following form
\begin{equation}\begin{split}
\E\left[   e^{\frac{1}{\gamma}  \int  \psi  X \dd  g  } Z_0^{1-\frac{c_{\psi}}{\gamma^2}}  \right]&=e^{-\gamma^{-2}H(\psi)}\E\left[ Z_0 e^{ -\frac{c_{\psi}}{\gamma^{2}} \left( \ln \int :e^{\gamma X}: e^{-h_{\psi}}  \dd \mu_{\psi}- \gamma \int (X-h_{\psi})\dd \mu_{\psi}\right)} e^{\frac{1}{\gamma}Y_{\psi}-\frac{1}{2\gamma^2}\var(Y_{\psi})} \right]\\
&=e^{-\gamma^{-2}H(\psi)}\E\left[ (\smallint :e^{\gamma X}: e^{h_{\psi}}w\dd    g)  e^{ -\frac{c_{\psi}}{\gamma^{2}} \left( \ln \int :e^{\gamma X}:  \dd \mu_{\psi}- \gamma \int X\dd \mu_{\psi}\right)} \right]
\end{split}
\end{equation}
where in the last line we used Cameron Martin formula \eqref{girsanov} and $\bE [Y_{\psi}X(x)]=h_{\psi}(x)$.

\medskip

By Proposition \ref{expintegra}, the quantity in the integral is bounded in $L^2$ (as the product of two quantities which are bounded in $L^4$)
and moreover it converges in probability when $\gamma$ tends  to zero to  
$$(\smallint e^{h_{\psi}}w\dd     g)  e^{ -c_{\psi} \left( \ln \int :X^2:  \dd \mu_{\psi}- (\int X\dd \mu_{\psi})^2\right)}.$$
This is enough to conclude our proof. \qed

\section{Uniform integrability}

In this section we  always consider $\mu$ to be probability measure on $\bbC$ and denote by  $\rho$ its 
density with respect to our reference measure $  g(x)\dd^2 x$. We  assume that \eqref{normalization} holds for for some $\eta>0$.

\subsection{Proof of Proposition \ref{expintegra} }
%
%

The proof of Proposition \ref{expintegra} requires a few technical estimates which 
we present now and prove at the end of the section.
The first one allows us to assert that the second term in our chaos expansion is 
uniformly integrable.

\begin{lemma}\label{expintegralimit}
Assuming that the probability measure $\mu$ is supported on $B(0,1):=\{ x \ : \ |x|\le 1\}$  satisfies  \eqref{normalization} then for any $\alpha>0$ we have
\begin{equation}\label{eqexpintegralimit}
\sup_{t \ge 0}  \E\left[  e^{- \alpha  [ \int  :\tilde{X}_t^2: \dd \mu  -  (\int \tilde{X}_{t} \dd \mu )^2     ]   }      \right]  < \infty.
\end{equation}
\end{lemma}

\begin{lemma}\label{lemmanegtaive}
Assuming that the probability measure $\mu$ is supported on $B(0,1):=\{ x \ : \ |x|\le 1\}$,
and satisfies $\inf_{B(0,1)} \rho>0$, then we have for some constant $C$ (which may depend on $\mu$), for every $\gamma\le 1$ and $\beta \ge 1$,
\begin{equation*}
  \E \left [  \left(\int :e^{\gamma\tilde X}: \dd \mu\right)^{-\frac{\beta}{\gamma^2}}   \right ]   \leq  e^{\frac{C \gb^2}{\gamma^2}}.
\end{equation*}
As a consequence we have, 
\begin{equation}\label{devs}
 \bbP\left[  \int :e^{\gamma \tilde X}: \dd \mu  \le \gamma \right]
 \le e^{-\frac{|\ln \gamma|^2}{4C\gamma^2}}
\end{equation}

\end{lemma}

\begin{proof}[Proof of Proposition \ref{expintegra}]

The almost sure convergence
follows from the expansion of the Wick exponential which is valid for $\gamma$ sufficiently small (it is valid for $X_\gep$ and both sides converge when $\gep$ tends to $0$),
\begin{equation}
 \smallint :e^{\gamma X}: \dd \mu  := 1+\sum_{k=1}^{\infty}\frac{\gamma^k}{k!} \smallint :X^k: \dd \mu. 
\end{equation}

\medskip

For practical reason, in the proof of \eqref{labound} we wish to reduce our domain of integration to $B(0,1)$. This can be achieved by splitting the sphere in two and considering each half separately.
Set   $q:=  \int_{|x|\le 1}\rho(x)g(x) \dd^2 x$
and let 
$$\mu_1(\dd^2 x):=  q^{-1}\mu(\dd x)\ind_{\{|x|\le 1\}} \text{ and } \mu_2(\dd^2 x):=  (1-q)^{-1}\mu(\dd^2 x)\ind_{\{|x|\ge 1\}}.$$
By using the concavity of $\ln$ and $ab\le q a^{\frac{1}{q}} +  (1-q)b^{\frac{1}{1-q}}$ we obtain 

\begin{equation}\label{decompro}
 \E\left[  e^{-\frac{\alpha}{\gamma^2}  (  \ln \int :e^{\gamma X}: \dd \mu  - \gamma \int X \dd \mu     )    }   \right] 
 \le  q \E[  e^{-\frac{\alpha}{\gamma^2}  (  \ln \int :e^{\gamma X}: \dd \mu_1  - \gamma \int X \dd \mu_1     )    }   ]+(1-q)
   \E[  e^{-\frac{\alpha}{\gamma^2}  (  \ln \int :e^{\gamma X}: \dd \mu_2  - \gamma \int X \dd \mu_2     )    }   ].
 \end{equation}

Now observe that the distribution $X$ 
is invariant by the transformation $x\mapsto (1/x)$, so the second term 
remains unchanged if we replace $\mu_2$ by $\mu_3$, the image measure of $\mu_2$ by the transformation $x\mapsto (1/x)$. We have 
$$\mu_3(\dd^2 x)= (1-q)^{-1} \rho(1/x) \frac{\dd^2 x}{|x|^4}\ind_{\{x\le 1\}}.$$
We can observe that the function $(1-q)^{-1} \rho(1/x)g(1/x) \frac{\dd^2 x}{|x|^4}$ 
satisfies \eqref{normalization}, and hence both terms in the r.h.s. of \eqref{decompro} can be treated in the same manner.
Also note that as $\tilde X= X+Z$  (see the construction of Section \ref{decompex}) 
where $Z$ is a Gaussian random variable (which is not independent of $X$),
replacing $X$ by $\tilde X$ does alter the value of the function inside the expectation by a lot.
More precisely noticing that 
$$:e^{\gamma  X}:= e^{-\gamma Z}:e^{\gamma \tilde X}:e^{\gamma^2 v}
 \quad \text{ with } \quad v(x):= \frac{1}{2}\ln 2- \frac{1}{2\pi}\int^{2\pi}_{0} G(x,e^{i\theta}) \dd \theta $$
where the term $v(x)$ accounts for the covariance between $X$ and $Z$ and an extra variance term,
we obtain that 
\begin{equation}
  \ln \int :e^{\gamma  X}: \dd \mu_1  - \gamma \int  X \dd \mu_1=       \ln \int :e^{\gamma \tilde X}: \dd \mu_1  - \gamma \int \tilde X \dd \mu_1 +\gamma^2\min_{x\in B(0,1)} v(x).
\end{equation}
Hence to prove \eqref{labound} it is sufficient to prove 
\begin{equation}\label{newbound}
 \sup_{\gamma\in (0,1]}\E\left[  e^{-\frac{\alpha}{\gamma^2}  (  \ln \int :e^{\gamma \tilde X}: \dd \mu_1  - \gamma \int \tilde X \dd \mu_1     )    }   \right] <\infty.
\end{equation}
For the rest of the proof we set $t=t_{\gamma}:=\gamma^{-1/8}$. We are first going to show that \eqref{newbound} holds with $\tilde X$ replaced by $\tilde X_{t_{\gamma}}$.

%

First, recalling 
 the definition of Wick exponential \eqref{supertaylor} using that 
$\E[\tilde{X}_t^2 (x)]=t+2$, 
we have by Jensen inequality
\begin{equation}\label{nozab}
   \ln \int :e^{\gamma \tilde{X}_t}:  \dd \mu_1 - \gamma \int \tilde{X}_t \dd \mu_1   \geq -\frac{\gamma^2}{2}(t+2).
\end{equation}
We introduce the event  
$$\mathcal{A}= \left\{\int |\tilde X_t(x)|^2 \ind_{|\tilde X_t(x)> t^2|} \dd \mu_1 \ge  e^{-t} \right\}.$$
Our idea is that on $\cA$ we can use Taylor expansion to get rid of $:\exp:$ and $\ln$ while the complement has such a small probability that 
a rough estimate will be sufficient.
A simple application of Markov inequality implies (recall that $\var(\tilde X_t)=t+2$) that 
\begin{equation}\label{bad}\bbP[\cA^{\complement}]\le e^{t} \bbE\left[\int |\tilde X_t(x)|^2 \ind_{\{\tilde X_t(x)> t^2\}} \dd \mu_1 \right]\le e^{- ct^3}. \end{equation}
We are going to prove that if $\gamma$ is sufficiently small on the event $\cA$, we have 
\begin{equation}\label{zab}
 \ln \int :e^{\gamma \tilde{X}_t}:  \dd \mu_1 -\gamma \int  \tilde X_t \dd \mu_1 \ge 
 \frac{\gamma^2}{2} \left[ \int :\tilde X_t^2:  \dd \mu_1 - 
 \left(\int \tilde X_t \dd \mu_1 \right)^2-1 \right].
\end{equation}
Using the formula $e^{u}\ge 1+u+\frac{u^2}{2}
+\frac{u^3}{6}$ for $u= \gamma \tilde X_t- \frac{\gamma^2}{2} (t+2)$,
we obtain that for some constant $C>0$, for all $\gamma\le 1$ and $t\ge 1$,  we have as soon as $|\tilde X_t(x)|\le t^2$, 
\begin{equation}
 :e^{\gamma \tilde{X}_t(x)}:\ge 1+\gamma \tilde X_t(x)
 + \frac{\gamma^2}{2}:\tilde X_t(x)^2: - C \gamma^3 t^6.
\end{equation}
Hence integrating we obtain 
 \begin{equation}
 \int :e^{\gamma \tilde{X}_t}:  \dd \mu_1
 \ge \mu_1(\{ x \ : \ |\tilde X_t(x)|\le t^2 \})
 +  \int (\gamma \tilde X_t
 + \frac{\gamma^2}{2} :\tilde X_t^2:) \ind_{\{|\tilde X_t(x)|\le t^2\}} \dd \mu_1 -  C \gamma^3 t^6.
 \end{equation}
Note that with our choice of $t$ the last term is smaller than $\gamma^2/8$ for small values of $\gamma$.
Now, on the event $\cA$, using  that $:\tilde X_t^2:\ge - (t+2)$ almost surely, as a consequence of the event's definition, the missing parts in the integral are negligible and we have thus for $\gamma$ sufficiently small
\begin{equation}
 \int :e^{\gamma \tilde{X}_t}:  \dd \mu_1
 \ge 1+ \gamma \int  \tilde X_t \dd \mu_1
 + \frac{\gamma^2}{2}  \int :\tilde X_t^2:  \dd \mu_1 - \gamma^2/4.
\end{equation}
Now the event $\cA$ guarantees that the integral terms on the right hand side are at most of respective order $t^2\gamma$ and 
$t^4\gamma^2$. Using this information together with the inequality $\ln(1+u)\ge u-\frac{u^2}{2}- |u|^3$ which is valid when $|u|$ is sufficiently small, we obtain \eqref{zab}. 

\medskip

\noindent
Now combining \eqref{nozab} and \eqref{zab} we obtain that 
\begin{equation}\label{fttt}
 \E\left[  e^{-\frac{\alpha}{\gamma^2}  (  \ln \int :e^{\gamma \tilde X_t}: \dd \mu_1  - \gamma \int \tilde X_t \dd \mu_1     )    }   \right]
 \le e^{\alpha (t+2)}\bbP[\cA^{\complement}]
 + \E\left[  e^{-\frac{\alpha}{2} \left[ \int :\tilde X_t^2:  \dd \mu_1 - 
 \left(\int \tilde X_t \dd \mu_1 \right)^2-1 \right]} \ind_{\cA}  \right] .
\end{equation}
The first term can be controlled using \eqref{bad}  and the second  using Lemma \ref{expintegralimit}. We conclude that 
\begin{equation}\label{newbound2}
 \sup_{\gamma\in (0,1]}\E\left[  e^{-\frac{\alpha}{\gamma^2}  (  \ln \int :e^{\gamma \tilde X_{t_{\gamma}}}: \dd \mu_1  - \gamma \int \tilde X_{t_{\gamma}} \dd \mu_1     )    }   \right] <\infty.
\end{equation}
Now to prove \eqref{newbound} with $\tilde X$
we set $$
\mathcal{B} = \left\{  \frac{\int :e^{\gamma \tilde X_{t_\gamma}}:\dd \mu_1}{\int :e^{\gamma \tilde X}: \dd \mu_1}\le 1+ \gamma^2  \right\}
$$
and bound separately the contribution of $\cB$ and its complement.
Using the decomposition 
\begin{multline}
   \ln \int :e^{\gamma \tilde X}: \dd \mu_1  - \gamma \int \tilde X \dd \mu_1     \\
=  \left(\ln \int :e^{\gamma \tilde X_ {t_\gamma}}: \dd \mu_1  - \gamma \int \tilde X_ {t_\gamma} \dd \mu_1 \right)  - \gamma \int (\tilde X- \tilde X_ {t_\gamma}) \dd \mu_1  -\log \left(\frac{\int :e^{\gamma \tilde X_ {t_\gamma}}:\dd \mu_1}{\int :e^{\gamma \tilde X}: \dd \mu_1}\right)
\end{multline}
and observing that the last term is smaller than $\gamma^2$ on $\cB$ we have (in the second line we just use $ab\leq a^2/2+b^2/2$)
 \begin{multline}
 \E\left[  e^{-\frac{\alpha}{\gamma^2}  (  \ln \int :e^{\gamma \tilde X}: \dd \mu_1  - \gamma \int \tilde X \dd \mu_1     )    } \ind_{\cB}  \right] 
\le e^{\alpha} \E\left[  e^{-\frac{\alpha}{\gamma^2}  (  \ln \int :e^{\gamma \tilde X_{t_\gamma}}: \dd \mu_1  - \gamma \int \tilde X_{t_\gamma} \dd \mu_1     )  + \frac{\alpha}{\gamma} \int (\tilde X- \tilde X_{t_\gamma}) \dd \mu_1  } \right]\\ 
\le \frac{e^{\alpha}}{2}\left( \E\left[  e^{-\frac{2\alpha}{\gamma^2}  (  \ln \int :e^{\gamma \tilde X_{t_\gamma}}: \dd \mu_1  - \gamma \int \tilde X_{t_\gamma} \dd \mu_1     )} \right] + \bbE \left[ e^{ \frac{2\alpha}{\gamma} \int (\tilde X- \tilde X_{t_\gamma}) \dd \mu_1  } \right] \right).
 \end{multline}
The first term  is bounded uniformly in $\gamma>0$, cf.\ \eqref{newbound2}, while for the second one, it is sufficient to observe that $U_{t_\gamma}=\int (\tilde X- \tilde X_{t_\gamma})\dd \mu_1$
is a Gaussian whose variance is small, the following being valid for some $c>0$, as a consequence of \eqref{normalization} and H\"older inequality
$$ \bbE[U_{t_\gamma}^2]= \int_{{t_\gamma}}^{\infty}  \left(\int Q_u(x,y) \mu_1(\dd^2  x)\mu_1(\dd^2  y)\right) \dd u \le e^{ -c {t_\gamma}}.$$

For the other part we have using H\"older's inequality
\begin{multline}
 \E\left[  e^{-\frac{\alpha}{\gamma^2}  (  \ln \int :e^{\gamma \tilde X}: \dd \mu_1  - \gamma \int \tilde X \dd \mu_1     )    } \ind_{\cB^{\complement}}  \right] \le \E\left[ \left(\int :e^{\gamma \tilde X}: \dd \mu_1 \right)^{\frac{-3\alpha}{\gamma^2}} \right]^{1/3} \E\left[ e^{\frac{3 \alpha }{\gamma} \int \tilde{X} \dd \mu_1   }   \right]^{1/3} \P(\mathcal{B}^{\complement})^{1/3}.
\end{multline}
Lemma \ref{lemmanegtaive} (applied to $\mu_1$) implies that the first term in the r.h.s.\ is smaller than
$e^{C (1+\alpha^2) \gamma^{-2}}$, while the second one is equal to 
$e^{\frac{3\alpha^2}{2\gamma^2} \var (\int \tilde{X} \dd \mu_1 ) }$.
To conclude it is sufficient to show that  
$\P(\mathcal{B}^\complement)\leq  e^{-\frac{ |\ln \gamma|}{C\gamma^2} }$.
Let us notice that 
\begin{equation}
 \P(\mathcal{B}^\complement)\le \P\left[\int :e^{\gamma \tilde X}: \dd \mu_1 \le \gamma\right]+ \P\left[ \int (:e^{\gamma \tilde X_{t_\gamma}} :-:e^{\gamma \tilde X}:) \dd \mu_1 \ge \gamma^3\right].
\end{equation}
The first term can be controlled by Lemma \ref{lemmanegtaive}.
As for the second one, its smallness is a consequence of the following result 
 proved in \cite{lacoin} under slightly different assumptions for $\mu$. The proof adapts however to this context, we replicate it in below  for the sake of completeness
\begin{lemma}\label{finallemma}
Given $\mu$  satisfying \eqref{normalization} there exists a constant $c>0$ such that for all $\gamma$ sufficiently small, we have for $t_{\gamma}=: \gamma^{-1/8}$,
\begin{equation*}
\bbP \left[ \int \left( :e^{\gamma\tilde X_{t_\gamma}}:-:e^{\gamma \tilde X}: \right) \dd \mu \ge e^{-t_\gamma/8} \right] \le e^{-c   \gamma^{-2-1/8}}.
\end{equation*}
\end{lemma}
Of course, as $e^{-t_\gamma/4}\leq \gamma^3$, this completes the proof. \end{proof}

\subsection{Proof of auxiliary Lemmas}

\begin{proof}[Proof of Lemma \ref{expintegralimit}] 
Setting  $N_t:=\int_{|x| \leq 1} \tilde{X}_{t} \dd \mu$ we have
\begin{equation}
 \E\left[  e^{- \alpha  [ \int  :\tilde{X}_{t}^2: \dd \mu  -  N_t^2     ]   }      \right]  
 =   \E\left[  e^{- \alpha  [ \int  :(\tilde{X}_{t}(x)-N_t)^2: \dd \mu     ]   }      \right] e^{\alpha \bbE[N^2_t]}.
\end{equation}
As  $\bbE[N^2_t]$ is uniformly bounded in $t$ it is sufficient to control the first term in the r.h.s.\ .
Let us set 
$$\tilde Y_{t}:= \tilde{X}_{t}(x)-N_t \text{ and }
\tilde Y_{[t_1,t_2]}=\tilde Y_{t_2}-\tilde Y_{t_1}.$$ 
Fixing $t_0$ (its exact value which depends on $\alpha$ and $\rho$ is to be chosen later), we assume that $t>t_0$.
Using orthogonality of the increments and the identity $abc\le \frac{1}{3}(a^3+b^3+c^3)$ we have 
\begin{multline}\label{prout}
  \E\left[  e^{- \alpha   \int  :\tilde{Y}^2_{t}: \dd \mu      } \right]\le \E\left[  e^{- \alpha  [ \int  (:\tilde{Y}^2_{t_0}: + :\tilde{Y}^2_{[t_0,t]}:+ 2 \tilde Y_{t_0}\tilde{Y}_{[t_0,t]} ) \dd \mu      ]   } \right] \\
  \le \frac{1}{3}\left(\E\left[  e^{- 3\alpha   \int  :\tilde{Y}^2_{t_0}: \dd \mu  }\right] +\E\left[e^{- 3\alpha   \int :\tilde{Y}^2_{[t_0,t]}: \dd \mu    }\right] +\E\left[e^{- 6\alpha   \int \tilde Y_{t_0}\tilde{Y}_{[t_0,t]} \dd \mu  } \right]\right).
  \end{multline}
The first term is easily controlled since we have for some constant $C(\rho)$ for every $x\in B(0,1)$
\begin{equation}
  :\tilde Y_{t_0}(x):^2 \ge - \bbE[ Y_{t_0}(x)^2 ]\ge -(t_0+C(\rho)).
\end{equation}
As for the two other terms, we rely on  \cite[Theorem 6.7]{Jans}
which states in particular that for some universal constant $c_2$ any square integrable variable 
$Z_{[t_0,t]}$ which can be expressed as the $L_2$ limit of second degree polynomials in $(X_t(x))_{t\ge 0, x\in B(0,1)}$, we have
\begin{equation}
 \forall t\ge 2, \quad \bbP[ Z\ge t\|Z\|_2]\le e^{-c_2 t},
\end{equation}
with $\|Z\|_2=\E[Z^2]^{1/2}$.
Applying this to $Z_{1}:=\int_{|x| \leq 1} :\tilde{Y}^2_{[t_0,t]}:\dd \mu $ and $Z_{2}:= \int_{|x| \leq 1} \tilde Y_{t_0}\tilde{Y}_{[t_0,t]} \dd \mu$,
we can bound the second and third in the r.h.s. of \eqref{prout} uniformly provided we can prove that for every $t\ge t_0$, we have $$\|Z_i\|_2< c_2/(12 \alpha), for  \text{ for } i=1,2. $$  
Using the notation $\cov_Y(x,y)=\bbE[Y(x)Y(y)]$ for the covariance functions we have 
\begin{equation}
 \begin{split}
  \|Z_1\|^2_2&= 2\int_{|x|, |y|\le 1}  (\cov_{\tilde Y_{[t_0,t]}}(x,y))^2 \rho(x)\rho(y)g(x)g(y)\dd^2 x \dd^2 y ,\\
\|Z_2\|^2_2&= \int_{|x|, |y|\le 1} \cov_{\tilde Y_{[t_0,t]}}(x,y) \cov_{\tilde Y_{t_0}}(x,y)  \rho(x)\rho(y)g(x)g(y)\dd^2 x \dd^2 y.
 \end{split}
\end{equation}
Tedious but standard calculation allows to show that for some positive constant $C$ (depending on the function $\rho$)
\begin{equation}
 \begin{split}
 |\cov_{\tilde Y_{t_0}}(x,y)|&\le  \log  C|x-y| ,\\
  |\cov_{\tilde Y_{[t_0,t]}}(x,y)|&\le \int_{t_0}^{\infty} Q_u(x,y)\dd u+ Ce^{-t_0/C}.
 \end{split}
\end{equation}
These estimates are sufficient to show that $\|Z_1\|_2$ and $\|Z_2\|_2$ 
can be made arbitrarily large by choosing $t_0$ large. \end{proof}

\begin{proof}[Proof of Lemma \ref{lemmanegtaive}]

With our positive assumption for $\rho$, at the cost of a multiplicative factor $e^{C \frac{\beta}{\gamma^2}}$ we can replace 
$\int :e^{\gamma \tilde X}: \dd \mu$ by $\int :e^{\gamma \tilde X}: \dd^2 x$.
Then we obtain the result by a simple comparison with the $1d$ log correlated case on the circle (well defined for $\gamma<\sqrt{2}$) for which we have an explicit expression.

Indeed, if $X_1(e^{i\theta})$ is the circular GFF with covariance $\E[ X_1(e^{i\theta})X_1(e^{i\theta'}) ]= \ln \frac{1}{|e^{i\theta}-e^{i\theta'}|}$ then
the Fyodorov-Bouchaud formula (proved by Remy \cite{Remy}) and the use of Stirling's asymptotics for the $\Gamma$-function yields (in our range of parameters)
\begin{equation}\label{steam}
 \E \left [  \left( \int_{0}^{2 \pi}   :e^{\gamma X_1( e^{i \theta})}: 
\dd\theta \right)^{-\frac{\beta}{\gamma^2}}      \right ] 
= \Gamma\left(1+\frac{\beta}{2}\right) \Gamma\left(1-\frac{\gamma^2}{2}\right)^{\frac{\beta}{\gamma^2}}  (2 \pi)^{-\frac{\beta}{\gamma^2}}
\le e^{C \left(\gb \log \gb + \gb \gamma^{-2}\right)}.
\end{equation}
The following holds
\begin{equation*}
\forall \rho,\rho' \geq \frac{1}{2}, \: \forall \theta,\theta', \quad  \E[\tilde{X}(\rho e^{i \theta}) \tilde{X}(\rho' e^{i \theta'}) ]   \leq \E[X_1(e^{i\theta}) X_1(e^{i\theta'})]+ \ln 8
\end{equation*} 
since $4 |\rho e^{i \theta}- \rho' e^{i \theta'}| \geq  |e^{i \theta}- e^{i \theta'}|$. Therefore, on the annulus $A= \lbrace x; \: \frac{1}{2} \leq |x| \leq 1\rbrace$ one can apply Kahane's inequality (see \cite[Theorem 2.1]{review}) to the convex function $x \mapsto x^{-\frac{\beta}{\gamma^2}}$. Letting $Y$ be a centered Gaussian with variance $\ln 8$ independent of $X_1$ we obtain that
\begin{align*}
   \E \left [  \left(\int_A :e^{\gamma \tilde X}: \dd \mu\right)^{-\frac{\beta}{\gamma^2}}     \right ]   
 & =   \E \left [  \left(\int_{\frac{1}{2}}^1 \int_{0}^{2 \pi}   :e^{\gamma \tilde{X}(\rho e^{i \theta})}:
 \rho d\rho d\theta \right)^{-\frac{\beta}{\gamma^2}}      \right ]    \\ 
& \leq  \E \left [  \left(\int_{\frac{1}{2}}^1 \int_{0}^{2 \pi}   :e^{\gamma X_1(e^{i \theta})}: 
 \rho d\rho d\theta \right)^{-\frac{\beta}{\gamma^2}}      \right ]  \bbE\left[ \left(:e^{\gamma Y}:\right)^{-\frac{\beta}{\gamma^2}}  \right] \\ 
& = e^{\frac{\beta^2 \log 8}{2\gamma^2}+ \frac{\beta\log 2}{\gamma^2}+ \frac{\beta}{2}} \E \left [  \left( \int_{0}^{2 \pi}   :e^{\gamma X_1( e^{i \theta})}: 
\dd\theta \right)^{-\frac{\beta}{\gamma^2}}      \right ] 
\end{align*}
and combined with \eqref{steam} gives us the desired estimate.
The estimate \eqref{devs} is obtained by a standard application of Markov inequality for $\gb=|\log \gamma|^2/(2C).$
\begin{equation}
 \bbP\left[ \int :e^{\gamma \tilde X}: \dd \mu \le \gamma \right]
 \le \gamma^{-\frac{\gb}{\gamma^2}}\bbE\left[ \left(\int :e^{\gamma \tilde X}: \dd \mu \right)^{-\frac{\beta}{\gamma^2}}\right].
\end{equation}
\end{proof}

\begin{proof}[Proof of Lemma \ref{finallemma}]
We write $\bbE_s$ for the conditional expectation $\bbE[ \cdot \ | \ \cF_s]$ where $(\cF_s)_{s\ge 0} $ is the natural filtration associated with $\tilde X_s$, and with some abuse of notation 
$\bbP_s(A):=\bbE_s[\ind_{A}]$.
In what follows, we will write $t$ for $t_\gamma$.

For fixed $s\geq 0$, we set  $\phi(s):= \bbE_t \left[ e^{s \int \left(:e^{\gamma\tilde X_t}:-:e^{\tilde X}: \right)\dd \mu } \right]$ and we have

\begin{equation}\label{grouch}
 \bbP_t \left[ \int :e^{\gamma\tilde X_t}:-:e^{\gamma \tilde X}: \dd \mu \ge e^{-t/8} \right] \le \max\left(1, \phi(s) e^{-s  e^{-t/8} }\right)
\end{equation}

 The random function $\phi$ is almost surely differentiable and if $\bar X_t:= \tilde X-\tilde X_t$ and 
 $\bar K_t(x,y):= \int^{\infty}_t Q_u(x,y) \dd u$
 we have 
 \begin{multline}
  \phi'(s)=\bbE_t \left[ \int \left(:e^{\gamma\tilde X_t}:-:e^{\gamma \tilde X}: \right)\dd \mu  e^{s\int \left(:e^{\gamma\tilde X_t}:-:e^{\gamma \tilde X}: \right)\dd \mu } \right]
  \\
  = \bbE_t  \left[\int :e^{\gamma \tilde X_t(x)}: \left(e^{s\int \left(:e^{\gamma\tilde X_t}:-:e^{\gamma \tilde X}: \right)\dd \mu } -:e^{\bar X_t(x)}:e^{s\int \left(:e^{\gamma\tilde X_t}:-:e^{\gamma \tilde X}: \right)\dd \mu } \right) \mu(\dd^2 x) \right]  \\
  =\int :e^{\gamma \tilde X_t(x)}: \bbE_t  \left[e^{s\int \left(:e^{\gamma\tilde X_t}:-:e^{\gamma \tilde X}: \right)\dd \mu } -e^{s\int \left(:e^{\gamma\tilde X_t}:-e^{\gamma^2 \bar K_t(x,\cdot):}e^{\gamma \tilde X}: \right)\dd \mu } \right]   \mu(\dd^2 x)
  \end{multline}
where in the last line we used Girsanov formula \eqref{girsanov}.
Now rewriting the expectation in the integrand of the r.h.s.\
we have
\begin{multline}
 \bbE_t  \left[e^{s\int \left(:e^{\gamma\tilde X_t}:-:e^{\gamma \tilde X}: \right)\dd \mu }\left(1-e^{-s\int (e^{\gamma^2 \bar K_t(x,\cdot)}-1):e^{\gamma \tilde X}: \dd \mu } \right)\right] \\
 \le  s \bbE_t  \left[e^{s\int \left(:e^{\gamma\tilde X_t}:-:e^{\gamma \tilde X}: \right)\dd \mu }\int (e^{\gamma^2 \bar K_t(x,\cdot)}-1):e^{\gamma \tilde X}: \dd \mu  \right]
 \le s\phi(s)\int (e^{\gamma^2 \bar K_t(x,\cdot)}-1) \dd \mu.
\end{multline}
Where the last linw is obtained using  the FKG inequality \eqref{FKG} for the field $\bar X_t$ and  the increasing functions
$e^{s\int \left(:e^{\gamma\tilde X_t}:-:e^{\gamma \tilde X}: \right)\dd \mu }$ and $\int (e^{\gamma^2 \bar K_t(x,\cdot)}-1):e^{\gamma \tilde X}: \dd \mu$ whose $\bbE_t$ average are respectively $\phi(s)$ and $\int (e^{\gamma^2 \bar K_t(x,\cdot)}-1) \dd \mu$.

\medskip

Using our assumption \eqref{normalization},
one can check that there exists a constant $C$ such that for all $\gamma$ sufficiently small all $x$ and $t>0$,  
\begin{equation}
 \int (e^{\gamma^2 \bar K_t(x,y)}-1)  \mu(\dd y)\le C \gamma^2 e^{-t}.
\end{equation}
This yields 
\begin{equation}
  \phi'(s)\le\left[ C\gamma^2 e^{-t}  \int :e^{\gamma \tilde X_t}: \dd \mu \right]s\phi(s).
\end{equation}
Hence on the event 
$ \cA_t:= \left\{ \int :e^{\gamma \tilde X_t}: \dd \mu \le  2 e^{t/2} \right\}, $
we have $\phi(s)\le e^{ C\gamma^2 e^{-t/2} s^2}$.
Hence integrating \eqref{grouch} for $s= e^{3t/8}$ we obtain 
\begin{equation}
 \bbP_t \left[ \int :e^{\gamma\tilde X_t}:-:e^{\gamma \tilde X}: \dd \mu \ge e^{-t/4} \right] \le \bbP[\cA^{\complement}_t] +\exp(-e^{t/4}/2).
\end{equation}
Finally we have
\begin{equation}
  \int :e^{\gamma \tilde X_t}: \dd \mu \le e^{t/2} +   \int :e^{\gamma \tilde X_t}: \ind_{\{\tilde X_t> t\gamma^{-1}/2\}} \dd \mu.
\end{equation}
Using the inequality (recall that $\var \tilde X_t(x)=t+2$),  
$\bbE\left[ :e^{\gamma \tilde X_t}: \ind_{\{ \tilde X_t> t\gamma^{-1}/2\}}\right]\le e^{-\frac{t\gamma^{-2}}{10}},$
and thus $\bbP[\cA^{\complement}_t]\le  e^{-\frac{t\gamma^{-2}}{20}}$.\end{proof}

%
%
%
%
%
%
%
%

\appendix

\section{Appendix}

%


\subsection{Relation between the centered Liouville action $S_{L,(\chi_k,z_k)}$ and the Liouville action $S$}\label{app:equiv}

Recall that the centered Liouville action is defined on $H^1(\hat{\C})$ by the following expression:
\begin{equation*}
S_{L,(\chi_k,z_k)}(h)=  \frac{1}{4\pi}\int_{\C}( |\nabla_z h (z)|^2 + 4 \pi \Lambda e^{h(z)}w(z)    g(z))\dd^2z+   \frac{1}{4 \pi  }(4- \sum_{k=1}\chi_k) \int_{\C} h(z)    g(z) \dd^2z
\end{equation*}
where
\begin{equation*}
l(  (\chi_k,z_k) )= - \sum_{k=1}^n \chi_k(1-\frac{\chi_k}{4}) \ln g(z_k)  -  \frac{1}{2} \sum_{k \not = j }  \chi_k \chi_j G(z_j, z_k) 
\end{equation*}

Now, we prove Lemma \ref{lemmaequivalence} on the link between $S_{L,(\chi_k,z_k)}$ and $S_{(\chi_k,z_k)}$.

\proof

Recall the integration by parts formula 
\begin{equation} \label{JPP}
\int_D \partial_z F \dd^2 z   = \frac{i}{2}  \oint_C F(z) \overline{\dd z}
\end{equation}
where $C$ is the exterior contour of the domain $D$.

Recall that $\pi S_{(\chi_k,z_k)}$ is the limit of $ \pi S_\epsilon$ as $\epsilon$ goes to $0$ where
\begin{align}
  \pi &S_\epsilon (\phi)  \nonumber \\
  =& \int_{\C \setminus \cup_{k=1}^n B(z_k,\epsilon) \cup \lbrace |z| > \frac{1}{\epsilon}\rbrace} (  |\partial_z \phi |^2  + \pi \Lambda e^{\phi(z)})\dd^2z  - i \sum_{k=1}^n \frac{\chi_k}{2} \oint_{|z-z_k|=\epsilon}   \phi(z) \frac{\overline{\dd z}}{\bar{z}-\bar{z_k}}+ 2 i \oint_{|z|=\frac{1}{\epsilon}}   \phi(z) \frac{\overline{\dd z}}{\bar{z}}\nonumber \\
&+  \frac{\pi}{2} \sum_{k=1}^n \chi_k^2 \ln \frac{1}{\epsilon}+ 8 \pi \ln \frac{1}{\epsilon}  \label{expressionlemmaS}
\end{align}
where here the contour integrals $\oint$ are oriented counterclockwise.

We first consider the case $\phi= h+\varphi$ with $h$ smooth and $\varphi$ is the explicit function
\begin{equation}\label{defvarphihere}
\varphi(z)= \ln (  g(z)w(z)).
\end{equation}
We have using the integration by parts formula \eqref{JPP}   
\begin{align*}
&  \int_{\C \setminus \cup_{k=1}^n B(z_k,\epsilon) \cup \lbrace |z| > \frac{1}{\epsilon}\rbrace}   |\partial_z \phi |^2  \dd^2 z    \\
& = -\frac{i}{2} \sum_{k=1}^n \oint_{|z-z_k|=\epsilon}   \phi(z) \partial_{\bar{z}}  \phi(z)  \overline{\dd z} +  \frac{i}{2} \oint_{|z|=\frac{1}{\epsilon}}   \phi(z)  \partial_{\bar{z}}  \phi(z)  \overline{\dd z}-  \int_{\C \setminus \cup_{k=1}^n B(z_k,\epsilon) \cup \lbrace |z| > \frac{1}{\epsilon}\rbrace}   \phi (z) \partial_z \partial_{\bar{z}} \phi (z)  \dd^2 z    .
\end{align*}
Now, since $h$ is smooth we have the expansion $\partial_{\bar{z}}\phi(z)= - \frac{\chi_k}{2} \frac{1}{ \bar{z}-\bar{z_k}  }+O(1)$ as $z$ goes to $z_k$ and hence
\begin{equation*}
 \oint_{|z-z_k|=\epsilon}   \phi(z) \partial_{\bar{z}}  \phi(z)  \overline{\dd z}=  - \frac{\chi_k}{2} \oint_{|z-z_k|=\epsilon}   \phi(z)    \frac{\overline{\dd z}}{\bar{z}-\bar{z_k}} +o(1)
\end{equation*}
as $\epsilon$ goes to $0$.
Also, we have the expansion $\partial_{\bar{z}}\phi(z)= - \frac{2}{ \bar{z} }+o(\frac{1}{|z|})$ as $z$ goes to infinity hence  
\begin{equation*}
 \oint_{|z|=\frac{1}{\epsilon}}   \phi(z)  \partial_{\bar{z}}  \phi(z)  \overline{\dd z}=  -2 \oint_{|z|=\frac{1}{\epsilon}}   \phi(z)  \frac{\overline{\dd z}}{\bar{z}}+o(1).
\end{equation*}
Therefore we get up to $o(1)$ terms that
\begin{align*}
  \pi S_\epsilon (\varphi)   =& - \int_{\C \setminus \cup_{k=1}^n B(z_k,\epsilon) \cup \lbrace |z| > \frac{1}{\epsilon}\rbrace}  \phi (z) \partial_z  \partial_{\bar{z}} \phi (z)  \dd^2 z     - \frac{i}{4} \sum_{k=1}^n \chi_k \oint_{|z-z_k|=\epsilon}   \phi(z) \frac{\overline{\dd z}}{\bar{z}-\bar{z_k}}\\&+  i \oint_{|z|=\frac{1}{\epsilon}}   \phi(z) \frac{\overline{\dd z}}{\bar{z}} +  \frac{\pi}{2} \sum_{k=1}^n \chi_k^2 \ln \frac{1}{\epsilon}+ 8 \pi \ln \frac{1}{\epsilon}  .
\end{align*}

Now, we analyze each term $\oint_{|z-z_k|=\epsilon}   \phi(z) \frac{\overline{\dd z}}{\bar{z}-\bar{z_k}}$. We have as $\epsilon$ goes to $0$
\begin{align*}
& -  \chi_k\frac{i}{4}\oint_{|z-z_k|=\epsilon}   \phi(z) \frac{\overline{\dd z}}{\bar{z}-\bar{z_k}}  \\
& =   - \frac{ \chi_k}{4} \int_0^{2 \pi }  ( h( z_k + \epsilon e^{i \theta}) + \ln g  (z_k+ \epsilon e^{i \theta})  + \sum_{j=1}^n \chi_j G(z_j, z_k + \epsilon e^{i \theta})   ) \dd \theta    \\
& =  -  \chi_k \frac{\pi}{2} (h( z_k) + \ln g  (z_k)  + \sum_{j \not = k}  \chi_j G(z_j, z_k)  )   + \frac{\pi}{4} \chi_k^2 \ln g (z_k)-  \frac{\pi}{2} \chi_k^2 \ln \frac{1}{\epsilon} -\frac{\pi}{2} \chi_k^2 \kappa +o(1)\\
& =  -  \chi_k \frac{\pi}{2} h( z_k)   -  \chi_k \frac{\pi}{2} \sum_{j \not = k}  \chi_j G(z_j, z_k)    -\frac{\pi}{2}  \chi_k(1-\frac{\chi_k}{2}) \ln g (z_k)-  \frac{\pi}{2} \chi_k^2 \ln \frac{1}{\epsilon} -\frac{\pi}{2} \chi_k^2 \kappa +o(1).
\end{align*}
We also have as $\epsilon$ goes to $0$ that 
\begin{align*}
  i \oint_{|z|=\frac{1}{\epsilon}}   \phi(z) \frac{\overline{\dd z}}{\bar{z}}  
& =   \int_0^{2 \pi }  ( h( \frac{1} {\epsilon} e^{i \theta}) + \ln g  ( \frac{1}{\epsilon} e^{i \theta})  + \sum_{k=1}^n \chi_k G(z_k, \frac{1}{\epsilon} e^{i \theta})   ) \dd \theta    \\
& =  2 \pi h( \infty) +2 \pi \ln 4+ 8\pi \ln \epsilon  + 2 \pi \sum_{k=1}^n \chi_k G(z_k, \infty) +o(1)\\
& =  2 \pi h( \infty) +2 \pi \ln 4+ 8\pi \ln \epsilon  + 2 \pi \sum_{k=1}^n \chi_k  ( \frac{1}{2} \ln 2-\frac{1}{4} \ln g(z_k)+\kappa) +o(1).  
\end{align*}
Hence, we get the following expansion
\begin{align*}
&  - \frac{i}{4} \sum_{k=1}^n \chi_k \oint_{|z-z_k|=\epsilon}   \phi(z) \frac{\overline{\dd z}}{\bar{z}-\bar{z_k}}  + \frac{\pi}{2} \sum_{k=1}^n \chi_k^2 \ln \frac{1}{\epsilon} +  i \oint_{|z|=\frac{1}{\epsilon}}   \phi(z) \frac{\overline{\dd z}}{\bar{z}}  +8 \pi \ln \frac{1}{\epsilon}\\
& = 2 \pi h(\infty)- \frac{\pi}{2} \sum_{k=1}^n \chi_k h( z_k)  - \pi \sum_{k=1}^n \chi_k(1-\frac{\chi_k}{4}) \ln g(z_k)  -  \frac{\pi}{2} \sum_{k \not = j }  \chi_k \chi_j G(z_j, z_k)  +C+o(1) 
\end{align*}
where $C=2 \pi \ln 4+ 2 \pi \sum_{k=1}^n \chi_k  ( \frac{1}{2} \ln 2+\kappa)-\frac{\pi}{2} \kappa \sum_{k=1}^n \chi_k^2$.

We now analyse the term
\begin{equation*} 
- \int_{\C \setminus \cup_{k=1}^n B(z_k,\epsilon) \cup \lbrace |z| > \frac{1}{\epsilon}\rbrace}  \phi (z) \partial_z  \partial_{\bar{z}} \phi (z)  \dd^2 z  
\end{equation*}
 by identifying the contribution of $h$ and $\varphi$ separately in the sum $\phi=h+\varphi$ with \eqref{defvarphihere}. Using $\Delta_{  g} G(\cdot,z_k)=-2\pi(\delta_{z_k}-\tfrac{1}{4\pi})$   and $\Delta_{   g}\ln   g=-2$  we get
$$\Delta_{  g}   \varphi (z) = -2 +\frac{1}{2}\sum_{k=1}^n \chi_k\quad \text{on}\quad  \C \setminus \cup_{k=1}^n B(z_k,\epsilon) \cup \lbrace |z| > \frac{1}{\epsilon}\rbrace.$$
We deduce that
\begin{align*}
& - \int_{\C \setminus \cup_{k=1}^n B(z_k,\epsilon) \cup \lbrace |z| > \frac{1}{\epsilon}\rbrace}  \phi (z) \partial_z  \partial_{\bar{z}} \varphi (z)  \dd^2 z    \\
&  =- \frac{1}{4} \int_{\C \setminus \cup_{k=1}^n B(z_k,\epsilon) \cup \lbrace |z| > \frac{1}{\epsilon}\rbrace}  \phi (z)   \Delta_z   \varphi (z) \dd^2 z    \\
&  =- \frac{1}{4} \int_{\C \setminus \cup_{k=1}^n B(z_k,\epsilon) \cup \lbrace |z| > \frac{1}{\epsilon}\rbrace}  \phi (z)   \Delta_{  g}    \varphi (z) {  g} (z) \dd^2 z    \\
&  = ( \frac{1}{2}-\frac{1}{8}\sum_{k=1}^n \chi_k)  \int_{\C \setminus \cup_{k=1}^n B(z_k,\epsilon) \cup \lbrace |z| > \frac{1}{\epsilon}\rbrace}  \phi (z) {  g} (z) \dd^2 z    \\   
&  = ( \frac{1}{2}-\frac{1}{8}\sum_{k=1}^n \chi_k)  \int_{\C}  \phi (z) {  g} (z) \dd^2 z    +o(1)\\
&  = ( \frac{1}{2}-\frac{1}{8}\sum_{k=1}^n \chi_k)  \int_{\C}  h (z) {  g} (z) \dd^2 z    +( \frac{1}{2}-\frac{1}{8}\sum_{k=1}^n \chi_k)  \int_{\C}  \ln {  g} (z) {  g} (z) \dd^2 z   +o(1).
\end{align*}

We also have that  
\begin{align*}
&    -\int_{\C \setminus \cup_{k=1}^n B(z_k,\epsilon) \cup \lbrace |z| > \frac{1}{\epsilon}\rbrace} \varphi (z) \partial_z  \partial_{\bar{z}} h (z)      \dd^2 z     \\
&   = -\frac{1}{4} \int_{\C \setminus \cup_{k=1}^n B(z_k,\epsilon) \cup \lbrace |z| > \frac{1}{\epsilon}\rbrace} \varphi (z) \Delta_z h (z)      \dd^2 z     \\
&   = -\frac{1}{4} \int_{\C \setminus \cup_{k=1}^n B(z_k,\epsilon) \cup \lbrace |z| > \frac{1}{\epsilon}\rbrace}  (-4 G(z,\infty) +4(\kappa-\frac{1}{2} \ln 2)+ \sum_{k=1}^n\chi_k G(z_k,z))    \Delta_z h (z)      \dd^2 z    \\
&   = -\frac{1}{4} \int_{\C \setminus \cup_{k=1}^n B(z_k,\epsilon) \cup \lbrace |z| > \frac{1}{\epsilon}\rbrace}  (-4 G(z,\infty) +4(\kappa-\frac{1}{2} \ln 2)+ \sum_{k=1}^n\chi_k G(z_k,z))    \Delta_{  g}  h (z)     {  g} (z) \dd^2 z    \\
&   = -\frac{1}{4} \int_{\C}  (-4 G(z,\infty) +4(\kappa-\frac{1}{2} \ln 2)+ \sum_{k=1}^n\chi_k G(z_k,z))    \Delta_{  g}  h (z)     {  g} (z) \dd^2 z    +o(1)\\
&   = -\frac{1}{4} \int_{\C}  \Delta_{  g}  (-4 G(z,\infty) +4(\kappa-\frac{1}{2} \ln 2)+ \sum_{k=1}^n\chi_k G(z_k,z))    h (z)     {  g} (z) \dd^2 z    +o(1)\\
& = -\frac{1}{4}   \left (       8 \pi (h(\infty)   - \frac{1}{4 \pi}\int_\C  h(z) {  g} (z) \dd^2 z  ) -2 \pi \sum_{k=1}^n\chi_k ( h(z_k)-  \frac{1}{4 \pi}\int_\C  h(z) {  g} (z) \dd^2z  )  \right )  \\
& =       -2 \pi h(\infty)   +( \frac{1}{2} -\frac{1}{8} \sum_{k=1}^n \chi_k) (\int_\C  h(z) {  g} (z) \dd^2z)   +\frac{\pi}{2}\sum_{k=1}^n\chi_k  h(z_k) .
\end{align*}
Therefore, gathering the two above expressions, we get 
\begin{align*}
 \pi S (\varphi)    & =      \int_{\C}( |\partial_z h (z)|^2 + \pi \Lambda e^{h(z)}e^{ \sum_{k=1}^n \chi_k G(z_k,z)}  {  g} (z)) \dd^2 z   \\
& +( 1 -\frac{1}{4} \sum_{k=1}^n \chi_k) (\int_\C  h(z) {  g} (z)  \dd^2 z )   - \pi \sum_{k=1}^n \chi_k(1-\frac{\chi_k}{4}) \ln {  g} (z_k)  -  \frac{\pi}{2} \sum_{k \not = j }  \chi_k \chi_j G(z_j, z_k)  +C_\star ((\chi_k))
\end{align*}
 where $C_{\star}((\chi_k))=2 \pi \ln 4+ 2 \pi \sum_{k=1}^n \chi_k  ( \frac{1}{2} \ln 2+\kappa)-\frac{\pi}{2} \kappa \sum_{k=1}^n \chi_k^2+( \frac{1}{2}-\frac{1}{8}\sum_{k=1}^n \chi_k)  \int_{\C}  \ln {  g} (z) {  g} (z) \dd^2z $.

\vspace{0.1 cm}
Now, we treat the general case. We write $\phi=h+\varphi$ and  $\phi_\ast=h_\ast+\varphi$ ($\phi_\ast$ is the solution of the Liouville equation) where $\varphi$ is defined by \eqref{defvarphihere}. Using $\phi=(\phi-\phi^*)+\phi^*$, we have 
\begin{equation*}
 \pi S_\epsilon (\phi)=  \int_{\C \setminus \cup_{k=1}^n B(z_k,\epsilon) \cup \lbrace |z| > \frac{1}{\epsilon}\rbrace} (  |\partial_z (\phi-\phi_\ast) |^2  + \pi \Lambda (e^{\phi(z)}-e^{\phi_\ast(z)}) ) \dd^2 z + \pi S_\epsilon (\phi_\ast)+ S_\epsilon'
\end{equation*} where $S_\epsilon'$ can be expressed as a sum
\begin{align*}
 S_\epsilon'  
& = S_{1,\epsilon}' +S_{2,\epsilon}'   
\end{align*}
with
\begin{align*}
 S_{1,\epsilon}'=&2 \int_{\C \setminus \cup_{k=1}^n B(z_k,\epsilon) \cup \lbrace |z| > \frac{1}{\epsilon}\rbrace}   \partial_z (\phi-\phi_\ast) \partial_z \varphi \dd^2z\\&- i \sum_{k=1}^n \frac{\chi_k}{2} \oint_{|z-z_k|=\epsilon}   (\phi-\phi_\ast)(z) \frac{\overline{\dd z}}{\bar{z}-\bar{z_k}}+ 2 i \oint_{|z|=\frac{1}{\epsilon}}   (\phi(z)-\phi_\ast(z)) \frac{\overline{\dd z}}{\bar{z}}
\end{align*}
and
\begin{equation*}
 S_{2,\epsilon}'=2 \int_{\C \setminus \cup_{k=1}^n B(z_k,\epsilon) \cup \lbrace |z| > \frac{1}{\epsilon}\rbrace}   \partial_z (h-h_\ast) \partial_z h_\ast \dd^2z  .
 \end{equation*}
Now, we have the following convergence  
\begin{equation*}
 S_{2,\epsilon}' \underset{\epsilon \to 0}{\rightarrow} 2 \int_{\C}   \partial_z (h-h_\ast) \partial_z h_\ast  \dd^2 z   
 \end{equation*}
so we just have to deal with the $ S_{1,\epsilon}'$ term. By integration by parts we get (where $o(1)$ is with respect to $\epsilon$ going to $0$)
\begin{align*}
  S_{1,\epsilon}'   =&2 \int_{\C \setminus \cup_{k=1}^n B(z_k,\epsilon) \cup \lbrace |z| > \frac{1}{\epsilon}\rbrace}   \partial_z (\phi-\phi_\ast) \partial_z \varphi  \dd^2 z - i \sum_{k=1}^n \frac{\chi_k}{2} \oint_{|z-z_k|=\epsilon}   (\phi-\phi_\ast)(z) \frac{\overline{\dd z}}{\bar{z}-\bar{z_k}}\\ &+ 2 i \oint_{|z|=\frac{1}{\epsilon}}   (\phi(z)-\phi_\ast(z)) \frac{\overline{\dd z}}{\bar{z}}   \\
     =&- 2 \int_{\C \setminus \cup_{k=1}^n B(z_k,\epsilon) \cup \lbrace |z| > \frac{1}{\epsilon}\rbrace}   \partial_z (\phi-\phi_\ast) \partial_{\bar{z}}\partial_z \varphi  \dd^2 z  +o(1)\\
 =&- \frac{1}{2} \int_{\C \setminus \cup_{k=1}^n B(z_k,\epsilon) \cup \lbrace |z| > \frac{1}{\epsilon}\rbrace}   (h-h_\ast) \Delta_z \varphi  \dd^2 z  +o(1).
\end{align*}
Since on $\C \setminus \cup_{k=1}^n B(z_k,\epsilon) \cup \lbrace |z| > \frac{1}{\epsilon}\rbrace$, we have 
\begin{equation*}
\Delta_z \varphi= (-2+ \frac{1}{2}\sum_{k=1}^n \chi_k) {  g} (z)
\end{equation*}
this leads to 
\begin{equation*}
 S_{1,\epsilon}' \underset{\epsilon \to 0}{\rightarrow} (1-\frac{1}{4} \sum_{k=1}^n \chi_k ) \int_{\C} (h-h_\ast) {  g} (z) \dd^2 z   .
 \end{equation*}
Gathering the above considerations, we get 
\begin{multline*}
 \pi S_\epsilon (\phi)\underset{\epsilon \to 0}{\rightarrow} \pi S (\phi_\ast)+  \int_{\C} (  |\partial_z (h-h_\ast) |^2  + \pi \Lambda( e^{\phi(z)}-e^{\phi_\ast(z)} )) \dd^2 z \\+ 2 \int_{\C}   \partial_z (h-h_\ast) \partial_z h_\ast  \dd^2 z    + (1-\frac{1}{4} \sum_{k=1}^n \chi_k ) \int_{\C} (h-h_\ast) {  g} (z) \dd^2 z  
\end{multline*}
which proves identity \eqref{mainidentityapp}. 

\qed

\subsection{Existence of solutions to the Liouville equation}\label{app:exist}
Here, we give a short proof of the existence and uniqueness to the equation \eqref{equationdebase}.

Let $\psi$ be some function defined on the Riemann sphere. We introduce the functional $J_{\psi}$ on functions $h\in\bar H^1(\hat{\C})$ with vanishing mean on the sphere 
\begin{equation}\label{defJ}
J_\psi (h)= \frac{1}{4 \pi}  \int_{\C} |\nabla_{g} h(z)|^2\, g(z) \dd^2 z    - \int_{\C} (\psi(z)- m_{g}(\psi) ) h(z) \, g(z) \dd^2 z      + c_\psi \ln \int_{\C} w(z) e^{h(z)} \, g(z) \dd^2 z  
\end{equation}
where we set
\begin{equation*}
c_\psi  =  \int_{\C}    \psi(z) g(z)  \dd^2 z  + \sum_k \chi_k -4  
\end{equation*}
and
\begin{equation*}
m_{g}(\psi)=  \frac{1}{4 \pi}  \int_{\C}    \psi(z) g(z)  \dd^2 z  .
\end{equation*}


Recall that we have the following Moser-Trudinger inequality for all functions $h\in\bar H^1(\hat{\C})$ with vanishing mean on the sphere (see \cite{MR} for example):
\begin{equation}\label{MTineq}
 \ln \int_{\C} w(z) e^{h(z)} \, g(z) \dd^2 z   \leq    \frac{1}{16 \pi (1 \wedge \inf_{i} (1-\chi_k/2) )} \int_{\C} |\nabla_{g} h(z)|^2 g(z) \dd^2 z    .
\end{equation}
Therefore, the functional $J_\psi$ is bounded from below if and only if $c_\psi>- 4 (1 \wedge \inf_{i} (1-\chi_k/2) )$. In that case, the minimum solves the following equation:  

\begin{proposition}\label{sol:h}
Assume $ c_\psi>- 4 (1 \wedge \inf_{i} (1-\chi_k/2) )$. Then the equation
\begin{equation}\label{equationdebase}
\Delta_{g} h = -2 \pi (\psi- m_{g} (\psi ))+2 \pi c_\psi (\frac{ w e^{h}}{\int w e^{h}  g  }-\frac{1}{4 \pi})
\end{equation}
admits a unique solution $h_\psi$ with vanishing mean on the sphere.
\end{proposition}
 
%
%

\subsection{Convexity considerations} \label{Convexconsi}

\subsubsection{General considerations}

Recall that when $\textsc{f}$ is a function taking values in $]-\infty,\infty]$ on some Banach space $B$ then we can define   its lower semicontinuous enveloppe $\textsc{f}^{\text{sc}}$ by the following limit 
\begin{equation*}
\textsc{f}^{\text{sc}}(\lambda)=  \underset{\delta \to 0}{\lim}  \; \underset{\lambda' \in B(\lambda,\delta)}{\inf}\: \textsc{f}(\lambda').
\end{equation*} 
This lower semicontinuous enveloppe satisfies the following properties:
\begin{enumerate}
\item
If $\textsc{f}$ is convex then so is $\textsc{f}^{\text{sc}}$. 
\item
For all $\lambda$ there exists some sequence $(\lambda_n)_{n \geq 1}$ such that $\lambda_n$ converges to $\lambda$ and $\textsc{f}(\lambda_n)$ converges to $\textsc{f}^{\text{sc}}(\lambda)$.
\end{enumerate}

Now, if $\textsc{f}$ is a convex function taking values in $]-\infty,\infty]$ we introduce the Legendre transform $\textsc{f}^{\ast}$ by the formula
\begin{equation*}
\textsc{f}^{\ast}(x)= \sup_{\lambda \in B}  (<\lambda,x>- \textsc{f}(\lambda) ).
\end{equation*} 
By item 2 above one can easily see that $\textsc{f}^{\ast}=(\textsc{f}^{\text{sc}})^{\ast}$. Finally we recall the Fenchel-Moreau theorem
\begin{equation*}
\textsc{f}^{\ast \ast}= \textsc{f}^{\text{sc}}.
\end{equation*}
 
Hence we deduce the following lemma which we will need in the following:

\begin{lemma}
Let I be some convex and lower semi continuous function on $B^{\ast}$ and $\textsc{f}$ some convex function such that
\begin{equation*}
\textsc{f}^{\text{sc}}(\lambda)= \sup_{x \in B^{\ast}}  (  <\lambda,x>- I(x)  )
\end{equation*}
then we have the following identity: $I= \textsc{f}^{\ast}$.
\end{lemma}

\proof
We have $I^{\ast}=\textsc{f}^{\text{sc}}$ and therefore by the Moreau-Legendre theorem we get $I= I^{\ast \ast}=(\textsc{f}^{\text{sc}})^{\ast}=\textsc{f}^{\ast}$. \qed

\subsubsection{The Legendre transform of the Liouville action}

Recall that $c_\psi= \int_{\C} \psi(z) g(z) \dd^2 z   + \sum_{k=1}^n \chi_k -4$.
Now, we consider the Laplace functional $\textsc{f} (\psi)$ defined for all $\psi \in H^1(\hat{\C})$ by the formula

 \begin{equation*}
\textsc{f} (\psi) =
\begin{cases}
& - l_{(\chi_k,z_k)}+  c_\psi \ln \frac{c_\psi}{\Lambda} - c_\psi  -J_\psi(  h_\psi )   \; \text{where}  \; h_\psi \; \text{solves} \; \eqref{equationdebase} \;  \text{if} \;  c_\psi>0 \\
& \infty, \; \text{if} \;  c_\psi \leq 0 \\
\end{cases}
\end{equation*}

We have the following lemma:

\begin{lemma}
The lower semicontinuous enveloppe of $\textsc{f}$ has the following expression:
\begin{equation*}
\textsc{f}^{sc} (\psi) =
\begin{cases}
& - l_{(\chi_k,z_k)}+ c_\psi \ln \frac{c_\psi}{\Lambda} - c_\psi  -J_\psi(  h_\psi )   \; \text{where}  \; h_0 \; \text{solves} \; \eqref{equationdebase} \;  \text{if} \;  c_\psi \geq 0, \\
& \infty, \; \text{if} \;  c_\psi < 0 .
\end{cases}
\end{equation*}
\end{lemma}

%

\proof
In the proof, we denote   by $\textsc{Y}$ the function on the right-hand side of the lemma. We want to show that $\textsc{Y}=\textsc{f}^{sc}$.
It is clear that $\textsc{Y}(\psi)= \textsc{f}(\psi)$ if $c_\psi \not = 0$. We choose $\psi$ such that $c_\psi=0$. Let $\epsilon>0$. We have
\begin{equation*}
\textsc{f} (\psi+\epsilon)=- l_{(\chi_k,z_k)}+\int \psi h_\epsilon \dd g -\frac{1}{4 \pi} \int |\nabla h_\epsilon|^2\dd^2z - \epsilon \ln \left(  \int w e^{h_\epsilon} \dd g  \right)  +4 \pi \epsilon \ln \frac{4 \pi \epsilon}{\Lambda}- 4 \pi \epsilon
\end{equation*}
where $h_\epsilon$ minimizes
\begin{equation*}
J_\epsilon(h)= \int_{\R^2} |\nabla h|^2 \, \dd^2z-4 \pi \int_{\R^2} (\psi- m_{g}(\psi) ) h \, \dd g  +4 \pi \epsilon \ln \int_{\R^2} w e^h \, \dd g
\end{equation*}
among functions with vanishing mean $h\in \bar H^1(\hat{\C})$. Since $J_\epsilon(h_\epsilon)$ is bounded independently from $\epsilon$ we deduce that $(h_\epsilon)_\epsilon$ is sequentially (weakly) compact and also by Moser-Trudinger \eqref{MTineq} that $w e^{h}$ stays bounded in $L^1$. Therefore, we can go to the limit in  \eqref{equationdebase} and deduce that any limit $\mathbf{h}$ of a subsequence of $(h_\epsilon)$ satisfies
\begin{equation*}
\Delta_{g} \mathbf{h}= -2 \pi (\psi- m_{g} (\psi )).
\end{equation*} 
Hence, we deduce convergence of $(h_\epsilon)$ to $\mathbf{h}$ which solves the above equation. Now, we use the fact that $\int |\nabla \mathbf{h}|^2\,\dd^2z= 2 \pi \int (\psi- m_{g} (\psi )) \mathbf{h}\,\dd g $ to deduce that we have convergence in $H^1(\hat \C)$. Therefore $\textsc{Y}(\psi)= \underset{\epsilon \to 0}{\lim}\, \textsc{f}(\psi+\epsilon) $. This shows the result. \qed

Now we  show  that $\textsc{f}^{sc}$ is the Legendre transform of the Liouville action 
\begin{proposition}\label{prop:legendre} 
We have
\begin{equation*}
\textsc{f}^{sc} (\psi)= \sup_{h} \left ( \int_{\C} \psi(z) h(z) g(z) \dd^2 z   -S_{L, (\chi_k,z_k)} (h)  \right )- l_{(\chi_k,z_k)}
\end{equation*}
and therefore also $S_{L, (\chi_k,z_k)} = (\textsc{f}^{sc})^{\ast}=\textsc{f}^{\ast} $. 
\end{proposition}

\proof The supremum of the above proposition is the same (up to the $- l_{(\chi_k,z_k)}$ term) as the supremum of the following function $T(c,h)$ defined for $c \in \R$ and $h\in \bar H^1(\hat{\C})$  by
\begin{equation*}
T(c,h) = c_\psi c +\int_{\C} \psi(z) h(z) g(z) \dd^2 z  - \frac{1}{4\pi}  \int_{\C}  |\nabla  h(z) |^2   \dd^2 z    - \Lambda e^{c} \int_{\C} w(z) e^{h(z)} g(z) \dd^2 z   .
\end{equation*}

If $c_\psi<0$, one can fix $h$ and take the limit $c \to -\infty$ which shows that the supremum of $T(c,h)$ is infinity. Therefore, we suppose that $c_\psi \geq 0$. 
For fixed $h$, $c \mapsto T(c,h)$ is maximal for $\Lambda e^{c} \int_{\C} w(z) e^{h(z)} g(z) \dd^2 z       = c_\psi   $. This yields
\begin{equation*}
\sup_{c \in \R}  T(c,h) = c_\psi \ln \frac{c_\psi}{\Lambda}-c_\psi+\int_{\C} \psi(z) h (z) g(z) \dd^2 z  - \frac{1}{4\pi}  \int_{\C}  |\nabla h(z) |^2   \dd^2 z    - c_\psi \ln \left(  \int_{\C} w(z) e^{h(z)} g(z) \dd^2 z   \right)
\end{equation*}
Now, one can conclude by optimizing this expression on $h$. \qed

\medskip
Now, we introduce the following set called exposed points in the language of large deviation theory (see section 4.5.3 in \cite{dembo}):

\begin{equation*}
\mathcal{F}= \Big\{ h \in H^{-1}(\hat \C)   ; \; \exists \psi \in   H^{1}(\hat \C), \; \forall h' \not = h ,   \int \psi h g- S_{L}(h)  >  \int \psi h' g- S_{L}(h')       \Big\}.
\end{equation*}

In words, $\mathcal{F}$ is the set of $ h \in H^{-1}(\hat \C)$ with the property: there exists $\psi \in   H^{1}(\hat \C)$ such that 
\begin{equation*}
u \mapsto \int \psi(z) u(z) g(z) \dd^2 z   -S_{L, (\chi_k,z_k)} (u) 
\end{equation*} 
admits a unique maximum at $u=h$. The corresponding $\psi$ is called an exposed hyperplane. The following lemma shows that there are many exposed points:

\begin{lemma}\label{lemmaexposed}
All smooth $h \in  H^{-1}(\hat \C) $ are exposed points for $S_{L, (\chi_k,z_k)}$, i.e. if $h$ is smooth then there exists some $\psi$ in $H^1(\hat \C)$ such that the function $u \mapsto \int \psi(z) u(z) g(z) \dd^2 z   -S_{L, (\chi_k,z_k)} (u) $ admits a unique supremum at $u=h$. The corresponding exposed hyperplane $\psi$ satisfies the following property: there exists some $t>1$ such that $\textsc{f} (t \psi)< \infty$. Moreover, one has for all open set $G \subset H^{-1} (\hat \C)$ 
\begin{equation*}
\inf_{  h \in G \cap \mathcal{F}  }  S_{L, (\chi_k,z_k)} (h)= \inf_{  h \in G }  S_{L, (\chi_k,z_k)} (h) .
\end{equation*}

\end{lemma}

\proof Fix some smooth $h$. We can write this element as $h_0+\beta$ where $h_0 \in \bar H^{1} (\hat \C)$ has vanishing mean on the sphere. We set
\begin{equation*}
\alpha+\sum_k \chi_k-4= \Lambda e^{\beta} \int_{\C} w(z) e^{h_0(z)} g(z) \dd^2 z  
\end{equation*} 
and then
\begin{equation*}
\bar{\psi}=   -\frac{1}{2 \pi} \Delta_{g} h_0+(\sum_k \chi_k-4+\alpha)(\frac{w e^{h_0}}{\int w e^{h_0} g  }-\frac{1}{4 \pi}).
\end{equation*}
We set $\psi=\bar{\psi}+\frac{\alpha}{4 \pi}$. It is easy to to see that the function $u \mapsto \int \psi u g -I(u) $ has a unique supremum given by $h$. Moreover, one has $c_\psi=\alpha+\sum_k \chi_k-4 >0$ and therefore there exists $t>1$ such that $c_ {t \psi}= t \alpha+\sum_k \chi_k-4 >0$ hence $\textsc{f} (t \psi)< \infty$. Finally, the equality $\inf_{  h \in G \cap \mathcal{F}  }  S_{L, (\chi_k,z_k)} (h)= \inf_{  h \in G }  S_{L, (\chi_k,z_k)} (h)$ is standard.\qed

\end{document}